\documentclass{amsart}

\usepackage{appendix}
\usepackage{amsmath}
\usepackage{amsthm}
\usepackage{amssymb}
\usepackage{amsfonts}
\usepackage{bm}
\usepackage[shortlabels]{enumitem}
\usepackage[bookmarks=true,hyperindex,pdftex,colorlinks,citecolor=red, linkcolor=cyan]{hyperref}
\usepackage{centernot} 
\usepackage{marginnote} 
\usepackage[left=3.5cm,right=3.5cm,top=3cm,bottom=3cm]{geometry}
\usepackage{bbm}
\usepackage[all]{xy}
\usepackage{tikz}
\usepackage{comment}
\usetikzlibrary{arrows}
\usetikzlibrary{shapes,decorations}
\usetikzlibrary{positioning}
\usepackage{xcolor}

\usepackage{mathrsfs}

  \newcommand{\eq}[1][r]
   {\ar@<-3pt>@{-}[#1]
    \ar@<-1pt>@{}[#1]|<{}="gauche"
    \ar@<+0pt>@{}[#1]|-{}="milieu"
    \ar@<+1pt>@{}[#1]|>{}="droite"
    \ar@/^2pt/@{-}"gauche";"milieu"
    \ar@/_2pt/@{-}"milieu";"droite"}

\DeclareMathOperator{\dist}{dist}                           
\DeclareMathOperator{\lspan}{span}                          
\DeclareMathOperator{\aco}{aco}                             
\DeclareMathOperator{\supp}{supp}                           
\DeclareMathOperator{\rad}{rad}                             
\DeclareMathOperator{\diam}{diam}                           
\DeclareMathOperator{\Lip}{Lip}                             

\newcommand{\N}{\mathbb{N}}             
\newcommand{\R}{\mathbb{R}}             
\newcommand{\C}{\mathbb{C}}             
\newcommand{\K}{\mathbb{K}}             

\newcommand{\pare}[1]{\left({#1}\right)}                    
\newcommand{\set}[1]{\left\{{#1}\right\}}                   
\newcommand{\norm}[1]{\left\|{#1}\right\|}                  
\newcommand{\dual}[1]{{#1}^\ast}                            
\newcommand{\wscl}[1]{\overline{#1}^{\dual{w}}}             

\newcommand{\iten}{\ensuremath{\widehat{\otimes}_\varepsilon}} 
\newcommand{\pten}{\ensuremath{\widehat{\otimes}_\pi}}         

\newcommand{\lipfree}[1]{\mathcal{F}({#1})}                 
\newcommand{\F}{\mathcal{F}}                                
\newcommand{\FS}{\mathcal{FS}}                                
\newcommand{\lipfreesub}[2]{\mathcal{F}_{#1}({#2})}         
\newcommand{\lipnorm}[1]{\norm{#1}_L}                       
\newcommand{\ideal}{\mathcal{I}}                   
\newcommand{\idealsub}[2]{\mathcal{I}_{#1}({#2})}           
\newcommand{\hull}{\mathcal{H}}                    

\newcommand{\restricted}{\mathord{\upharpoonright}}

\def\<{\langle}
\def\>{\rangle}
\newcommand{\ep}{\varepsilon}

\newcommand{\re}{\mathrm{Re}}
\newcommand{\im}{\mathrm{Im}}

\theoremstyle{plain}
\newtheorem{theorem}{Theorem}[section]
\newtheorem{lemma}[theorem]{Lemma}
\newtheorem{corollary}[theorem]{Corollary}
\newtheorem{proposition}[theorem]{Proposition}

\theoremstyle{definition}
\newtheorem*{definition*}{Definition}
\newtheorem{definition}[theorem]{Definition}
\newtheorem{example}[theorem]{Example}

\theoremstyle{remark}
\newtheorem{remark}[theorem]{Remark}

\begin{document}

\title[A pre-adjoint approach on weighted comp. operators between Lip. spaces]{A pre-adjoint approach on weighted composition operators between spaces of Lipschitz functions}

\author[A. Abbar]{Arafat Abbar}

\author[C. Coine]{Cl\'ement Coine}

\author[C. Petitjean]{Colin Petitjean}

\address[A. Abbar]{LAMA, Univ Gustave Eiffel, Univ Paris Est Creteil, CNRS, F--77447, Marne-la-Vall\'ee, France}
\email{abbar.arafat@gmail.com}

\address[C. Coine]{Normandie Univ, UNICAEN, CNRS, LMNO, 14000 Caen, France}
\email{clement.coine@unicaen.fr}

\address[C. Petitjean]{LAMA, Univ Gustave Eiffel, Univ Paris Est Creteil, CNRS, F--77447, Marne-la-Vall\'ee, France}
\email{colin.petitjean@univ-eiffel.fr}

\date{} 




\maketitle

\begin{abstract}
We consider weighted composition operators, that is operators of the type $g \mapsto w \cdot g \circ f$, acting on spaces of Lipschitz functions. Bounded weighted composition operators, as well as some compact weighted composition operators, have been characterized quite recently. In this paper, we provide a different approach involving their pre-adjoint operators, namely the weighted Lipschitz operators acting on Lipschitz free spaces. This angle allows us to improve some results from the literature. Notably, we obtain a distinct characterization of boundedness with a precise estimate of the norm. We also characterise injectivity, surjectivity, compactness and weak compactness in full generality.
\end{abstract}

\section{Introduction}

Let $(M,d)$ be a metric space and let $\K = \R$ or $\C$. For every Lipschitz function $f : M \to \K$, we let
$$
\|f\|_L :=  \sup_{x \neq y \in M} \frac{|f(x)-f(y)|}{d(x,y)}.
$$
Denote by $\Lip(M, \K)$ the Banach space of bounded Lipschitz maps from $M$ to $\K$ equipped with the norm
$$
\forall f \in \Lip(M,\K), \quad \|f\| = \max(\|f\|_{\infty} , \|f\|_L).
$$
We often omit the letter $\K$.
Note that in the literature, the preferred norm is sometimes
$\|f\|_\infty + \|f\|_L$ which turns $\Lip(M)$ into a Banach algebra. These two norms are of course equivalent so the choice does not matter for our purposes. 
The pros and cons of each norm are discussed at page 35 in \cite{Weaver2}.

Next, let $0_M \in M$ be a distinguished point of $M$. We let $\Lip_0(M,\K) = \Lip_0(M)$ be the $\K$-vector space of Lipschitz maps from $M$ to $\K$ vanishing at $0_M$. Equipped with the norm $\| \cdot \|_L$,
$\Lip_0(M)$ is a Banach space. It is sometimes called the \textit{pointed Lipschitz space} over $M$. It is important to recognize that $\Lip_0$ spaces are actually more general than $\Lip$ spaces. Indeed, every $\Lip$ space is in fact, in every meaningful respect, a $\Lip_0$ space (see Section~\ref{Section1.2} for more details). For this reason, we will mainly focus on $\Lip_0$ spaces, and then derive the corresponding results for $\Lip$ spaces as simple corollaries.

Now let us define the objects we will study in this paper. Let $f : M \to N$ be a function between two metric spaces and let $w : M \to \mathbb{K}$ be any map (\textit{``the weight''}). Then the \textit{weighed composition operator} $wC_f : \Lip_0(N) \to \K^M$ is defined by
$$ \forall g \in \Lip_0(N), \forall x \in M, \quad  wC_f(g)(x) = w(x) g \circ f(x).$$
A complete description of bounded weighted composition operators from $\Lip_0(N)$ to $\Lip_0(M)$ has recently been given in \cite{GolMa22} (see also \cite{DA} in the case when $M$ is bounded). 
In the same paper \cite{GolMa22}, under the hypothesis that $f(M)$ it totally bounded in $N$, the authors provide a necessary and sufficient condition for such operators to be compact (see also \cite{DA} for a similar result with the additional assumption $w \in \Lip(M)$). Finally, injectivity and surjectivity of such operators are characterized in \cite{DA} (see also \cite{GolMa16} for compact $M$), while weak compactness is considered in \cite{Gol18}. In the latter paper, $M$ is assumed to be a compact metric space such that the little Lipschitz space has the uniform separation property (these metric spaces are characterized in \cite[Theorem~A]{AGPP22}). The present paper revolves around these properties from a different angle. Indeed, in the aforementioned articles, the fact that $\Lip_0(M)$ is a dual space is implicitly used at various places (especially for the results dealing with compactness). We intend to use this property in a more obvious manner.
\smallskip

The space $\Lip_0(M)$ has indeed a natural predual which we describe now. For $x\in M$, we let $\delta(x) \in \Lip_0(M)^*$ be the evaluation functional defined by $\<\delta(x) , f \> = f(x), \ \forall f\in \Lip_0(M).$ It is readily seen that $M$ is isometric to a subset of $\Lip_0(M)^*$ via the isometry $\delta : M \to \Lip_0(M)^*$. The Lipschitz free space over $M$ is the Banach space
    $$\F(M) := \overline{ \mbox{span}}^{\| \cdot  \|} \delta(M) = \overline{ \mbox{span}}^{\| \cdot  \|}\left \{ \delta(x) \, : \, x \in M  \right \} \subset \Lip_0(M)^*.$$
Then, one can check that $\F(M)^*$ is isometrically isomorphic to $\Lip_0(M)$ via the duality brackets
$$
\left\langle g, \delta(x) \right\rangle = g(x), \ \forall x\in M, \ \forall  g\in \Lip_0(M).
$$
Lipschitz free spaces have been studied for a couple decades, and keep attracting a lot of attention. We refer to \cite{GK, Weaver2} for general information. In the literature, they are mostly defined and studied for the scalar field $\K = \R$, but we will explain in Section~\ref{Section2} that there is no real reasons not to consider the complex scalars case. 

\noindent \begin{minipage}{0.60\linewidth}
    The cornerstone of our study is the next fundamental extension property: if $f : M\to N$ is a Lipschitz map between two pointed metric spaces, there exists a unique (bounded) linear operator $\widehat{f} : \F(M) \to \F(N)$ such that $\widehat{f} \circ \delta_M = \delta_N \circ f$, and moreover $\|\widehat{f}\| = \|f\|_L$. 
\end{minipage}
\begin{minipage}{0.05\linewidth}
    ~
\end{minipage}
\begin{minipage}{0.25\linewidth}
    $$\xymatrix{
		M \ar[r]^f \ar[d]_{\delta_{M}}  & N \ar[d]^{\delta_{N}} \\
		\F(M) \ar[r]_{\widehat{f}} & \F(N).
	}$$
\end{minipage}

\noindent Adapting the definition of $\widehat{f}$ above, one can define a linear map $w\widehat{f}  : \mbox{span} \left( \delta(M) \right) \to \F(N)$, by setting
$$
\ w\widehat{f}\Big(\sum_{i=1}^n a_i\delta(x_i)\Big) = \sum_{i=1}^n a_i w(x_i) \delta(f(x_i)).
$$
When this map is bounded, it has a unique extension still denoted by
$
w\widehat{f}  : \F(M) \to \F(N)
$. Such a bounded operator is named \textit{weighted Lipschitz operator}. It turns out that the adjoint of $w\widehat{f}$ corresponds exactly to the weighted composition operator $wC_f : \Lip_0(N) \to \Lip_0(M)$. Indeed, one has for every $g \in \Lip_0(N)$ and $x \in M$:
$$ \<(w\widehat{f})^*(g) , \delta(x) \> = \< g , w\widehat{f}(\delta(x)) \> = \< g , w(x) \delta(f(x))\> = w(x) g(f(x)) = wC_f(g)(x).$$
Thus $(w\widehat{f})^* = wC_f$ and therefore we may use the general facts which connect the properties of an operator to the properties of its adjoint. For completeness, let us recall some classical relations that we will use in the sequel: if $T:X \to Y$ is a bounded operator between Banach spaces then
\begin{enumerate}[$(i)$]
    \item $T^*$ is injective  if and only if $T$ has dense range.
    \item $T^*$ is surjective if and only if $T$ is an isomorphism onto $T(X)$. 
    \item $T^*$ is compact  if and only if $T$ is compact.
    \item $T^*$ is weakly compact  if and only if $T$ is  weakly compact.
\end{enumerate}
The first two items can be found in \cite{Fabian} (Exercises 2.36 and 2.39 at page 58). Assertion $(iii)$ is Schauder's theorem (see e.g. \cite[Theorem~3.4.15]{Megg}) and 
$(iv)$ is Gantmacher's theorem (see e.g. \cite[Theorem~3.5.13]{Megg}). 
\smallskip

Hence, in this paper we shall study $w\widehat{f}$ and deduce the desired properties of $wC_f$ as simple consequences. We obtain in this way some refinements of the results mentioned above by dropping as many assumptions as possible on $M$ and $w$. At this point, we wish to point out that this strategy was followed in \cite{ACP21} for the simpler case $w=1$. It allowed the authors to remove the assumptions of separability and boundedness of $M$ in the characterisation of compactness obtained in \cite{Vargas1}. 
\smallskip

We now describe the content of this paper. We provide in Section~\ref{Section2} a detailed introduction to Lipschitz free spaces. We emphasize the case of complex scalars which is generally not considered in the literature. Notably, the notion of support is extended to this complex version of Lipschitz free spaces.  In Section~\ref{bounded}, we give several necessary and sufficient conditions for $w\widehat{f}$ to be bounded (see Theorem~\ref{Poids}). In particular, we retrieve the conditions from \cite[Theorem 2.1]{GolMa22}.  We also provide a characterisation of injectivity and surjectivity of the weighted composition operators $wC_f$ in Proposition~\ref{injectivity} and Proposition~\ref{surjectivity}, respectively.  
Section~\ref{sectioncompact} deals with compactness and weak compactness of weighted Lipschitz operators and weighted composition operators. A key point is the equivalent wording provided in Proposition~\ref{caracCompactWeight}, which in turns is an adaptation of \cite[Theorem 2.3]{Vargas2}. Another important ingredient is Theorem~\ref{ThmNormToWeak} which allows us to prove that weak compactness is actually equivalent to (norm) compactness for the class of operators we consider in this article.
The latter result in an improvement of \cite[Theorem B]{ACP21} where $w=1$,  and \cite[Theorem 2.1]{Gol18} where $M$ is a compact purely 1-unrectifiable metric space. The two key elements mentioned above permit us to deduce a new and general characterization of (weak) compactness for weighted Lipschitz/composition operators. Unfortunately, it is tedious to obtain a satisfactory statement without further assumptions, so we postpone this result to the appendix (see Theorem~\ref{CaracCompactgeneral}). Finally, in subsections~\ref{sectionsuffcompac} and \ref{sectioncompacw=1}, we focus on compactness but this time with stronger assumptions on either $M$, $f : M \to N$ or $w: M \to \K$. The outcome is that we are able to provide nicer statements than Theorem~\ref{CaracCompactgeneral}. Importantly, we retrieve the main results from \cite{GolMa22} and \cite{DA}.

\bigskip

\subsection{Notation and background.}~

If $F$ is a finite set, we denote by $|F|$ its cardinality. If $X$ is a Banach space, we let $X^*$ be its topological dual space and $B_X$ (resp. $S_X$) be its closed unit ball (resp. its unit sphere). Next, $\aco (S)$ denotes the absolute convex hull of a subset $S \subset X$. If $S$ is a subset of $X$, its annihilator is defined by
$$S^{\perp}:=\{x^\ast\in X^\ast:\, \langle x^\ast,x \rangle=0,\, \forall x\in S\}.$$
Note that $S^{\perp}$ is a closed subspace of $X^\ast$. Moreover, for a subset $S$ of $X^\ast$,
$$S_{\perp}:=\{x\in X:\, \langle x^\ast,x \rangle=0,\, \forall x^\ast\in S\}$$
is a closed subspace of $X$.
For another Banach space $Y$, $\mathcal{L}(X,Y)$ stands for the space of bounded linear operators $T : X\to Y$. If $X=Y$, we simply write $\mathcal{L}(X)$. As usual, $X \pten Y$ is the projective tensor product of the Banach spaces $X$ and $Y$, that is, the completion of $X \otimes Y$ equipped with the norm
$$
\|z\|_{\pi} := \inf \left\lbrace  \sum \|x_i\| \|y_i\| \right\rbrace ,
$$
where the infimum runs over all finite families $(x_i)_i$ in $X$ and $(y_i)_i$ in $Y$ such that $z = \sum_i x_i \otimes y_i$. We recall the following isometric identification of its dual space:
$$
(X \pten Y)^* = \mathcal{L}(X,Y^*).
$$
The isometry is given by the linear map taking any functional $u : X\otimes Y \to \mathbb{K}$ to the operator $\varphi_u : X \to Y^*$ defined for every $x \in X$ and $y\in Y$ by $\left\langle \varphi_u(x),y \right\rangle = u(x \otimes y)$. See for instance for more details. 
\medskip

Throughout the paper, $M$ and $N$ will be pointed metric spaces and the distinguished points will be denoted by $0_M$ and $0_N$ respectively, or simply $0$ if there is no ambiguity.
For future reference, let us state here a lemma which is a simple consequence of McShane extension's theorem, see e.g. \cite[Theorem 1.33 and Corollary 1.34]{Weaver2}. In fact, a concrete formula can be given to extend a Lipschitz map between metric spaces. This result allows us to separate two or more points of $M$ by an element of $\text{Lip}_0(M,\R)$ (therefore also an element of $\text{Lip}_0(M,\C)$).

\begin{lemma}
	\label{LemmaConstrctionLipF}
	Let $M$ be a pointed metric space, let $p\in M\setminus \{0_M\}$ and let $\ep \in (0, d(p,0_M)/4)$. Then there exists $f \in \text{Lip}_0(M,\R)$ such that $f = 1$ on $B(p, \ep)$ and $f=0$ on $M \setminus B(p, 2 \ep).$
\end{lemma}

Finally, if $w : M \to \K$ is a weight map, then $\text{coz}(w) := \{x \in M \; : \; w(x) \neq 0\}$.

\subsection{\texorpdfstring{Lip versus Lip$_\mathbf{0}$}{Lip versus Lip0}} \label{Section1.2}
This short section is based on Section~2.2 in \cite{Weaver2}. We include the main highlights for convenience of the reader. 

\begin{itemize}[leftmargin =*, itemsep = 5pt]
    \item First of all, let $(M,d)$ be any metric space. Then one can define a new metric by setting $\rho=\min(d,2)$. This turns $M$ into a bounded metric space with $\diam(M) :=\sup\{\rho(x,y) \; : \; x \neq y \in M\} \leq 2$. It is then rather easy to see that $\Lip(M,d)$ and $\Lip(M, \rho)$ are the same spaces, meaning that they are actually the same sets of functions with the same norms (for more details, see \cite[Proposition~2.12]{Weaver2}). Therefore, when working with $\Lip(M,d)$, we may assume that $\diam(M) \leq 2$.
    \item Now if $(M,\rho)$ is any metric space of diameter less than 2, then we artificially add a point $e$ to $M$: $M^e:=M \cup\{e\}$. This point plays the role of the distinguished point in $M$. We then extend $\rho$ to a metric $d^e$ on $M^e\times M^e$ by setting $d^e(x,e)=1$ for every $x\in M$. Now $\Lip(M)$ is naturally identified with $\Lip_0(M^e)$ (see \cite[Proposition 2.13]{Weaver2}). The isometry is given by $g \in \Lip(M) \mapsto g^e \in \Lip_0(M^e)$ where $g^e\restricted_M = g$ and $g^e(e)=0$.

\end{itemize}
As a result, every $\Lip$ space is in fact, in every meaningful respect, a $\Lip_0$ space. 
This fact will allow us to obtain, for free, results for operators between Lip spaces using the corresponding results for operators between Lip$_0$ spaces. For future reference, in the following numbered proposition we combine the previous elements and apply them to weighted composition operators.

\begin{proposition} \label{Lip0toLip}
    For metric spaces $M,N$ we let $M^e = M \cup\{e\}$ and $N^e = N \cup \{e\}$ be the metric spaces as constructed above. Let us fix some maps $w : M \to \K$ and $f : M \to N$. We extend $f$ and $w$ by defining $f(e)=e$ and $w(e)=0$, and we denote $f^e : M^e \to N^e$ and $w^e : M^e \to \K$ these extensions. Then $wC_f$ is a bounded operator from $ \Lip(N)$ to $\Lip(M)$ if and only if $w^eC_{f^e}$ is a bounded operator from $\Lip_0(N^e)$ to $\Lip_0(M^e)$. In such a case these operators are conjugated and have the same norm. In particular, $wC_f$ is compact if and only if $w^eC_{f^e}$ is compact.
\end{proposition}


\section{On the complex version of Lipschitz free spaces.}\label{Section2}

\subsection{Definition}

Let $(M,d)$ be a pointed metric space and $\K = \R$ or $\C$. If $X$ is Banach space over $\K$ then $\Lip_0(M, X)$ stands for the $\K$-Banach space of Lipschitz maps from $M$ to $X$ vanishing at $0_M$, equipped with the norm:
$$\displaystyle
\forall f \in \Lip_0(M,X), \quad \|f\|_L :=  \sup_{x \neq y \in M} \frac{\|f(x)-f(y)\|_X}{d(x,y)}.$$ 
Notice that $\|f\|_L$ is simply the best Lipschitz constant of $f$. Recall that for $x\in M$, we let $\delta_{\K}(x) \in \Lip_0(M,\K)^*$ be the evaluation functional defined by $\<\delta_{\K}(x) , f \> = f(x), \ \forall f\in \Lip_0(M, \K).$ The map $\delta_{\K} : M \to \F(M,\K)$ is an isometry. 
\begin{itemize}
    \item The Lipschitz free space over $M$ is the Banach space
    $$\F(M,\R) := \overline{ \mbox{span}}^{\| \cdot  \|}\left \{ \delta_{\R}(x) \, : \, x \in M  \right \} \subset \Lip_0(M,\R)^*.$$
    \item In a similar way, the complex version of the Lipschitz free space over $M$ is defined by:
    $$\F(M,\C) := \overline{ \mbox{span}}^{\| \cdot  \|}\left \{ \delta_{\C}(x) \, : \, x \in M  \right \} \subset \Lip_0(M,\C)^*.$$
\end{itemize}
The next fundamental property is often referred to as the ``universal extension property of Lipschitz free spaces''. 

\begin{theorem} \label{thm1}
Let $M$ be a pointed metric space. Let $X$ be a Banach space over \nolinebreak$\K$.
Then for every $f \in \Lip_{0}(M,X)$, there exists a unique $\overline{f} \in \mathcal{L}(\F(M, \K),X)$ such that:
\begin{enumerate}[label = (\roman*)]
\item $\| \overline{f} \|_{\mathcal{L}(\F(M,\K),X)} = \| f \|_L$
\item the following diagram commutes:

\begin{minipage}{0.4\linewidth}
$$\def\commutatif{\ar@{}[rd]|{\circlearrowleft}}
\xymatrix{
M \ar[r]^f \ar@{^{(}->}[d]_{\delta} & X \\
    \F(M,\K) \ar[ru]_{\overline{f}} 
}
$$
\end{minipage}
\begin{minipage}{0.4\linewidth}
i.e : $f=\overline{f} \circ \delta$.
\end{minipage}
\end{enumerate}

In particular, the next isometric identification holds: 
$$\Lip_0(M,X) \equiv \mathcal{L}(\F(M, \K),X).$$
\end{theorem}
A direct application (in the case $X = \K$) of the previous theorem provides another basic yet important information: 
$$\Lip_0(M,\K) \equiv \F(M, \K)^*.$$
Also, the weak$^*$ topology induced by $\F(M,\K)$ on $\Lip_0(M,\K)$ coincides with the topology of pointwise convergence on norm-bounded subsets of $\Lip_0(M, \K)$.
At this point we should mention that, contrary to the real case, the complex version of Lipschitz free spaces has not been considered very often in the recent literature. In fact, complex Lipschitz free spaces were studied in a systematical way in the first edition of the book of Weaver \cite{Weaver1} (where they are called Arens-Eells spaces). The second edition of the same book seems more interested by the real case (see the comments at pages 86 and 125 in \cite{Weaver2}). More recently, the complex version was defined explicitly as above in \cite{AACD20} where it is proved that for every metric space $M$, there exists a bounded metric space $B_M \subset \F(M,\K)$ such that $\F(M,\K)$ is isomorphic to $\F(B_M,\K)$. 
The next proposition is another easy consequence of Theorem~\ref{thm1}. It is generally cited as the ``linearization property of Lipschitz free spaces''.

\begin{proposition} \label{diagramfree}
	Let $M$ and $N$ be two pointed metric spaces. Let $f \colon M \to N$ be a Lipschitz map such that $f(0_M) = 0_N$. Then, there exists a unique bounded linear operator $\widehat{f} \colon \F(M,\K) \to \F(N,\K)$ with $\|\widehat{f}\|=\mathrm{Lip}(f)$ and such that the following diagram commutes:
	$$\xymatrix{
		M \ar[r]^f \ar[d]_{\delta_{M}}  & N \ar[d]^{\delta_{N}} \\
		\F(M, \K) \ar[r]_{\widehat{f}} & \F(N,\K).
	}$$
	More precisely, for every $\gamma=\sum_{i=1}^n a_i\delta(x_i)\in \F(M,\K)$,  
	$\widehat{f}(\gamma)=\sum_{i=1}^n a_i\delta(f(x_i))$.
\end{proposition}
Operators of the kind $\widehat{f} \colon \F(M,\K) \to \F(N,\K)$ will be called Lipschitz operators in this paper.

\subsection{The real case versus the complex case} \label{sectionFC}

Even if the definition of $\F(M,\C)$ is alike to that of $\F(M,\R)$, one should be careful since some features of $\F(M , \R)$ do not work equally well for $\F(M,\C)$. For instance, it is a well known fact that if $K \subset M$ contains the base point, then $\F(K,\R)$ is isometrically isomorphic to the subspace 
$$\F_M(K,\R):=\overline{\mbox{span}}^{\| \cdot  \|}\left \{ \delta_{\R}(x) \, : \, x \in K  \right \}$$
of $\F(M,\R)$. This is mainly due to the fact that every Lipschitz map defined on $K$ can be extended to $M$ without increasing its Lipschitz constant; this is the well-known McShane extension's theorem (see e.g. \cite[Theorem 1.33]{Weaver2}). In the case of complex scalars, an increasing of the Lipschitz constant by a factor $\frac{4}{\pi}$ is unavoidable; see \cite[Theorem 5.15]{Weaver2}. Therefore, $\F(K,\C)$ is only $\frac{4}{\pi}$-isomorphic to the subspace $\F_M(K,\C):=\overline{\mbox{span}}^{\| \cdot  \|}\left \{ \delta_{\C}(x) \, : \, x \in K  \right \}$. On the other hand, $\F(M,\R)$ embeds isometrically into $\F(M,\C)$ in a very natural way:

\begin{proposition} \label{RemarkValueRC}
	For every $(a_i)_{i=1}^n \in \R^n$ and every $(x_i)_{i=1}^n \in M^n$:
	$$\Big\|\sum_{i=1}^n a_i \delta_\R(x_i) \Big\|_{\F(M,\R)} =  \Big\|\sum_{i=1}^n a_i \delta_\C(x_i) \Big\|_{\F(M,\C).}$$
	Consequently,  the map $I : \sum_{i=1}^n a_i \delta_\R(x_i)  \mapsto \sum_{i=1}^n a_i \delta_\C(x_i)$ extends to a $\R$-linear isometry from $\F(M,\R)$ onto $\overline{\lspan}_{\R}(\delta_\C(M))\subset \F(M , \C)$. 
\end{proposition}

\begin{proof}
Let $\gamma= \sum_{i=1}^n a_i \delta_\R(x_i)$.
	One inequality is rather obvious:
	\begin{align*}
		\| \gamma  \|_{\F(M,\R)} &= \sup_{f \in B_{\Lip_0(M,\R)}} \Big|\sum_{i=1}^n a_i f(x_i) \Big| \\
		&\leq \sup_{f \in B_{\Lip_0(M,\C)}} \Big|\sum_{i=1}^n a_i f(x_i) \Big| 
		= \Big\|\sum_{i=1}^n a_i \delta_\C(x_i) \Big\|_{\F(M,\C).}
	\end{align*}
Conversely, using the Hahn-Banach theorem we may pick $f \in B_{\Lip_0(M,\C)}$ such that $\|I\gamma \|_{\F(M,\C)}=\langle f , I\gamma \rangle = \sum_{i=1}^n a_i f(x_i)$.
	Taking the real part in this equality gives
	$$ \|I\gamma \|_{\F(M,\C)} = \sum_{i=1}^n a_i \re(f(x_i)). $$
It is readily seen that the function $g : x \in M \mapsto \re(f(x)) \in \R$ is $1$-Lipschitz
	and 
	$$\|\gamma\|_{\F(M,\R)} \geq  \langle g , \gamma \rangle =  \sum_{i=1}^n a_i g(x_i)) = \|I\gamma\|_{\F(M,\C)}.$$
\end{proof}

Next, we wish to highlight that in the literature on vector-valued Lipschitz free spaces (e.g. \cite{GPR, Rueda}), the notation $\F(M,\C)$ refers to the projective tensor product $\F(M,\R) \pten \C$, where both factors are seen as real Banach spaces. This approach is motivated by the following $\R$-linear isometric identifications:
$$ \Lip_0(M,\C) \equiv \mathcal{L}(\F(M, \R),\C) \equiv (\F(M,\R) \pten \C)^*. $$
We will explain that both point of views are actually equivalent. Indeed, $\F(M,\R) \pten \C$ canonically becomes a complex vector space under the scalar multiplication 
$$ \forall \lambda  \in \C, \; \forall \gamma \otimes z \in  \F(M,\R) \otimes \C , \quad \lambda \cdot (\gamma \otimes z) := \gamma \otimes \lambda z.$$
Moreover the projective norm  $\|\cdot\|_{\pi}$ of $\F(M,\R) \pten \C$ is compatible with this $\C$-vector space structure, that is $(\F(M,\R) \pten \C , \|\cdot\|_{\pi})$ is a Banach space over $\C$. This can be seen in \cite[Lemma~1]{Bochnak} (see also \cite[Proposition~9]{Munoz} in the case of general tensor norms). The only point that requires care is the absolute homogeneity: For every $\gamma \in \F(M,\R) \pten \C$ and $\lambda \in \C^*$,
\begin{align*}
	\| \lambda \gamma\|_\pi &= \inf\left \{ \sum_{n=1}^{+\infty}  |\lambda_n| \,  \|\gamma_n\|_{\F(M,\R)}  \; : \; \lambda \gamma = \sum_{n=1}^{+\infty} \gamma_n \otimes \lambda_n \right\} \\
	&=\inf\left \{ |\lambda| \sum_{n=1}^{+\infty}  |\lambda_n\lambda^{-1}| \,  \|\gamma_n\|_{\F(M,\R)}  \; : \; \gamma = \sum_{n=1}^{+\infty} \gamma_n \otimes \lambda_n\lambda^{-1} \right\} \\
	&= |\lambda| \|\gamma\|_{\pi}.
\end{align*}
The complex Banach space $(\F(M,\R) \pten \C , \|\cdot\|_{\pi})$ described above is called the \textit{Bochner complexification} of $\F(M,\R)$. If follows rather directly from basic tensor product theory that its dual is identified with the injective tensor product $\Lip_0(M,\R) \iten \C$ equipped with the $\C$-vector space structure as described above. The latter $\C$-Banach space is called the \textit{Taylor complexification} of $\Lip_0(M,\R)$, and its norm is given by the formula:
$$\forall f_1 , f_2 \in \Lip_0(M,\R), \quad \|f_1 \otimes 1 + f_2 \otimes i \|_{\ep} = \sup_{\theta \in [0 ,2 \pi]} \|\cos(\theta)f_1 + \sin(\theta) f_2\|_{\Lip_0(M,\R)}.$$

\begin{remark} 
	If one prefers to avoid the use of tensor products, then the above considerations can be paraphrased as follows: $\F(M,\R) \times \F(M,\R)$ becomes a complex normed space when its linear structure and norm are defined for $\gamma_1 , \gamma_2 , \mu_1 , \mu_2 \in \F(M,\R)$ and $a, b \in \R$ by
	\begin{align*}
		(\gamma_1, \gamma_2) + (\mu_1, \mu_2) &:= (\gamma_1 + \mu_1, \gamma_2 + \mu_2) \\
		(a + ib)\cdot (\gamma_1, \gamma_2) &:= (a\gamma_1 - b\gamma_2, b\gamma_1 + a\gamma_2) \\
		\|(\gamma_1, \gamma_2)\|_\pi &:= \sup \big|\langle f_1 , \gamma_1 \rangle + \langle f_2 , \gamma_2 \rangle \big|, 
	\end{align*}
	where the last supremum is taken over all functions $f_1, f_2 \in \Lip_0(M,\R)$ such that 
 $$\|(f_1 , f_2)\|_{\ep} := \sup_{\theta \in [0 ,2 \pi]} \|\cos(\theta)f_1 + \sin(\theta) f_2\|_{\Lip_0(M,\R)} \leq 1.$$ 
 Observe that we clearly have 
	$$\max\{\|\gamma_1\| , \|\gamma_2\|\} \leq \|(\gamma_1, \gamma_2)\|_\pi \leq \|\gamma_1\| + \|\gamma_2\|.$$ 
\end{remark}

The next proposition shows that our definition of $\F(M,\C)$ leads to the same Banach space as the vector-valued approach presented above.

\begin{proposition} \label{pten}
If $M$ is a pointed metric space then $\Lip_0(M,\C) \equiv \Lip_0(M,\R) \iten \C$ and $\F(M,\C) \equiv \F(M,\R) \pten \C$ as $\C$-vector spaces.
\end{proposition}

\begin{proof}
	Let $T : \Lip_0(M,\R) \iten \C \to \Lip_0(M,\C)$ and $S :  \Lip_0(M,\C) \to  \Lip_0(M,\R) \iten \C $ be the $\C$-linear maps such that 
	$$T(f_1 \otimes 1 + f_2 \otimes i) := f_1 + i f_2 \quad \text{and} \quad S(f) := \re(f) \otimes 1 + \im(f)\otimes i.  $$
	Obviously $S \circ T = Id$, $T \circ S = Id$, and these maps are isometries since 
	\begin{align*}
		\|T(f_1  \otimes 1   + f_2 \otimes i)  \| & = \|f_1  + i f_2 \|_{\Lip_0(M,\C)} \\
		&= \sup_{x \neq y} \big| \big(f_1(x)-f_1(y)\big) + i\big( f_2(x) - f_2(y) \big) \big|   d(x,y)^{-1}\\
		&= \sup_{x \neq y} \Big( \big(f_1(x)-f_1(y)\big)^2 + \big( f_2(x) - f_2(y) \big)^2 \Big)^{\frac 1 2}   d(x,y)^{-1} \\ 
  		&= \sup_{x \neq y} \; \sup_{\theta \in [0 ,2 \pi]}  |\cos(\theta)(f_1(x) - f_1(y)) +  \sin(\theta)(f_2(x) - f_2(y))| d(x,y)^{-1} \\ 
		&= \sup_{\theta \in [0 ,2 \pi]} \|\cos(\theta)f_1 + \sin(\theta) f_2\|_{\Lip_0(M,\R)}\\ 
		&= \|f_1 \otimes 1 + f_2 \otimes i\|_{\ep}.
	\end{align*}
	Since $T$ and $S$ are continuous for the topology of pointwise convergence, the Banach-Dieudonné theorem implies that $T$ and $S$ are weak$^*$-to-weak$^*$ continuous. Therefore their pre-adjoint operators $T_* : \F(M,\C) \to \F(M,\R) \pten \C$ and $S_* : \F(M,\R) \pten \C \to \F(M,\C)$
	are isometric isomorphisms.
\end{proof}

It is not difficult to provide a concrete formula for $S_*$:
$$\forall \gamma_1 , \gamma_2 \in \F(M,\R), \quad  S_* (\gamma_1 \otimes 1 + \gamma_2 \otimes i) = I(\gamma_1) + i I(\gamma_2).$$
where $I : \F(M,\R) \to \F(M,\C)$ is the isometry from Proposition~\ref{RemarkValueRC}. To write an explicit (intrinsic) formula of the inverse mapping, we need the following terminology.

\begin{definition}
	For every $\gamma \in \F(M,\C)$, let 
	\begin{itemize}
		\item $\overline{\gamma} : f \in \Lip_0(M,\C) \mapsto \overline{\big\langle \overline{f} , \gamma \big\rangle} \in \C$,
		\medskip
		
		\item $\re(\gamma) = \frac 1 2 (\gamma + \overline{\gamma})$ and $\im(\gamma) = \frac{1}{2i} (\gamma - \overline{\gamma})$. 
	\end{itemize}
\end{definition} 
It is straightforward to check that $\overline{\gamma} \in \F(M,\C)$ with $\|\overline{\gamma}\| = \|\gamma\|$. Thus $\re(\gamma)$ and $\im(\gamma)$ belong to $\F(M,\C)$, and the triangle inequality gives $\|\re(\gamma)\| \leq \|\gamma\|$, $\|\im(\gamma)\| \leq \|\gamma\|$. We also clearly have $\gamma = \re(\gamma) + i \, \im(\gamma)$. In a more concrete way: if $\gamma = \sum_j (a_j+ib_j)\delta(x_j)$ then $\overline{\gamma} = \sum_j (a_j-ib_j)\delta(x_j)$, $\re(\gamma) = \sum_j a_j \delta(x_j)$ and $\im(\gamma) = \sum_j b_j \delta(y_j)$. In particular, we readily obtain that for every $\gamma \in \F(M,\C)$, 
$$\re(\gamma) , \im(\gamma) \in I(\F(M,\R)) = \overline{\lspan}_{\R}(\delta_\C(M)),$$
where $I$ is the isometry from Proposition~\ref{RemarkValueRC}.
Hence we can see $\re(\gamma)$ and $\im(\gamma)$ as elements of $\F(M,\R)$ (up to a composition with $I^{-1} : \overline{\lspan}_{\R}(\delta_\C(M)) \to \F(M , \R)$).
With this terminology, one can easily check that $T_* : \F(M,\C) \to \F(M,\R) \pten \C$ is the isometry given by 
$$\forall \gamma \in \F(M , \C), \quad T_*(\gamma) = I^{-1}(\re(\gamma))\otimes 1 +  I^{-1}(\im(\gamma))\otimes i. $$

\begin{remark} \label{RemarkQuotient}
	The inclusion mapping $Id : f \in \Lip_0(M,\R) \mapsto f \in \Lip_0(M,\C)$ is a $\R$-linear isometry. Since it is obviously continuous for the topology of pointwise convergence, the Banach-Dieudonné theorem implies that this map is weak$^*$-to-weak$^*$ continuous. Therefore its pre-adjoint operator is a $\R$-linear quotient map $Q$ from $\F(M,\C)$ (seen as real Banach space) to $\F(M,\R)$. One can check that $Q(\gamma) = \re(\gamma)$, $\forall \gamma  \in \F(M,\C)$. In particular $Q \circ I = Id_{\F(M,\R)}$. 
\end{remark}

\subsection{Support in complex Lipschitz free spaces}

Recall that if $K \subset M$ then
$$\F_M(K,\K):=\overline{\mbox{span}}^{\| \cdot  \|}\left \{ \delta_{\K}(x) \, : \, x \in K  \right \}.$$
Our next aim is to define a notion of support for elements $\gamma \in \F(M,\K)$. For $\K = \R$, this has been achieved in the paper \cite{support1} for bounded metric spaces and later in \cite{APPP2019} for unbounded metric spaces. The key point is the intersection theorem which is as follows:

\begin{theorem}[\textbf{Intersection theorem}]
Let $M$ be a complete pointed metric space and let $\{K_i:i\in I\}$ be a family of closed subsets of $M$ containing the base point. Then
$$
\bigcap_{i\in I}\F_M(K_i,\K)=\mathcal{F}_{M}\pare{\bigcap_{i\in I}K_i, \K} .
$$
\end{theorem}
Once the previous theorem is established, one can define the \textit{support of an element $\gamma \in \F(M,\K)$} to be the smallest closed subset $K \subset M$ such that $\gamma \in \F_M(K, \K)$. We claim that the proof of the intersection theorem presented in \cite{support1} for $\K=\R$ can be followed line by line to obtain its complex counterpart. To convince the reader, we will recall the main ingredients of the proof of the intersection theorem, arguing that there is absolutely no difference between the real and the complex case. 
\smallskip

Let us consider first a \textbf{bounded} and \textbf{complete} metric space $M$. Is it readily checked that
$\Lip_0(M,\C)$ is an algebra under pointwise multiplication:
$$\forall f,g \in \Lip_0(M,\C), \quad \|f \cdot g\|_L \leq 2\diam(M) \|f\|_L \|g\|_L.$$
Next, for any set $K \subset M$ containing the base point, we define 
$$ \mathcal I(K) := \{f \in \Lip_0(M,\C) \; : \; f\restricted_K = 0\}.$$
This is a weak$^*$-closed ideal of $\Lip_0(M,\C)$ \big(an ideal in $\Lip_0(M,\C)$ is a subspace $Y$ such that $f\cdot g\in Y$ for any $f\in Y$ and $g\in \Lip_0(M,\C)$\big). In fact, it is not hard to see that $\F_M(K,\C)^\perp = \ideal(K)$ and $\ideal(K)_\perp = \F_M(K,\C)$. On the other hand, for any subspace $Y \subset \Lip_0(M,\C)$ we will consider the hull of $Y$:
$$ \hull(Y):= \{x \in M \; : \; f(x) = 0 \text{ for all } f \in Y\}.$$
This is also not difficult to see that $\hull(\ideal(K)) = K$ for any closed subset $K \subset M$ (simply consider the 1-Lipschitz map $x \mapsto \dist(x , K)$). The following statement is the complex version of \cite[Proposition~3.2]{support1}.
\begin{proposition}\label{PropIdeal}
    Let $M$ be a bounded and complete pointed metric space. If $Y$ is an ideal in  $\Lip_0(M,\C)$ then $\overline{Y}^{w^*} = \ideal ( \hull(Y) )$.
\end{proposition}
The proof of \cite[Proposition~3.2]{support1} also works for complex scalars. Indeed, the main ingredients are:
\begin{itemize}[leftmargin=*, itemsep=5pt]
    \item \textit{If $Y$ is a norm closed ideal in $\Lip_0(M,\C)$ then $Y$ is weak$^*$-closed if and only if $Y = \ideal(\hull(Y))$ \cite[Corollary 4.2.6]{Weaver1}.} Notice that complex-valued Lipschitz maps are considered from the beginning of Chapter 4 in \cite{Weaver1}.
   
    \item \textit{Multiplication operators $M_h : \Lip_0(K,\C) \to \Lip_0(M,\C)$ and their pre-adjoint operators $W_h : \F(M,\C) \to \F(K,\C)$, sometimes called weighted operators.} Precisely, let $h\in\Lip(M)$ have bounded support. Let $K\subset M$ which contains the base point and the support of $h$. For $f\in\Lip_0(K,\C)$, $T_h(f)$ is the function given by
    \begin{equation*}
    \label{eq:T_h}
        T_h(f)(x)=\begin{cases}
        f(x)h(x) & \text{if } x\in K \\
        0 & \text{if } x\notin K.
    \end{cases} \,
    \end{equation*}
    Then $T_h$ defines a weak$^*$-to-weak$^*$ continuous linear operator from $\Lip_0(K,\C)$ into $\Lip_0(M,\C)$, and $\norm{T_h}\leq\norm{h}_\infty+\rad(\supp(h))\lipnorm{h}$.
    So there is an associated bounded linear operator $W_h\colon\lipfree{M,\C}\rightarrow\lipfree{K,\C}$ such that $\dual{W_h}=T_h$. 
    These objects work equally well for $\K = \R$ and $\K = \C$ and we refer the reader to \cite[Lemma~2.3]{APPP2019} for more details. 
    \item \textit{If $Y$ is an ideal then $\overline{Y}^{w^*}$ is also an ideal.} The proof of this claim uses an operator of the kind $W_h$ as defined above; see \cite[Proposition~3.2]{support1} for more details.
\end{itemize}

Exactly as in \cite[Theorem~3.3]{support1} (or \cite[Theorem~2.1]{APPP2019}), this is enough to deduce the intersection theorem for bounded metric spaces. Indeed, consider a family $\{K_i:i\in I\}$ of closed subsets of $M$ containing the base point. Then
define $Y=\lspan\set{\idealsub{M}{K_i}:i\in I}$. It is easily checked that $Y$ is an ideal. Therefore $\overline{Y}^{w^*} = \ideal ( \hull(Y) )$ thanks to Proposition~\ref{PropIdeal}. Writing $K=\bigcap_i K_i$, we claim that $K = \hull(Y)$. Indeed, the inclusion $K \subset \hull(Y)$ is obvious so we focus on the converse. If $x \not\in K$, then there exists $i \in I$ such that $x \not\in K_i$. We define the 1-Lipschitz map $f : x \in M \mapsto d(x, K_i)$ which belongs to $Y$. Clearly $f(x) \neq 0$ which means that $x \not\in \hull(Y)$, and therefore the reverse inclusion $\hull(Y) \subset K$ holds. In particular,
$\wscl{Y}=\idealsub{M}{K}$
and from there we conclude as follows:
\begin{align*}
\bigcap_{i\in I}\lipfreesub{M}{K_i , \C} &= \bigcap_{i\in I}\idealsub{M}{K_i}_\perp = \pare{\bigcup_{i\in I}\idealsub{M}{K_i}}_\perp = Y_\perp = \pare{\wscl{Y}}_\perp  \\
& = \idealsub{M}{K}_\perp = \lipfreesub{M}{K,\C} \,.
\end{align*}

\medskip

To obtain the intersection theorem for unbounded metric spaces $M$, we simply apply Section~7.2 in \cite{AACD20} where complex Lipschitz free spaces are considered from the beginning of the aforementioned paper. Now the intersection theorem allows us to define the notion of support as follows:
\begin{definition}
	The support of $\gamma \in \F(M,\C)$, denoted by $\supp(\gamma)$, is the smallest closed subset $K \subset M$ such that $\gamma \in \F_M(K, \C)$.
\end{definition} 

Note that by convention $\lspan \emptyset = 0$ and therefore $\supp (0) = \emptyset$. We conclude this section with a few results related to supports which will be useful in the sequel. The first one is the complex version of \cite[Proposition 2.7]{support1}.

\begin{proposition} \label{CharacSupport}
    Let $M$ be a complete pointed metric space and let $\gamma \in \F(M,\C)$. Then $x \in \supp(\gamma)$ if and only if for every neighbourhood $U$ of $x$, there exists $f \in \Lip_0(M,\C)$ whose support is contained in $U$ and such that $\langle f , \gamma \rangle \neq 0$. Moreover, in that case we may take $f \geq 0$.
\end{proposition}

We now relate the support of an element with the support of its real and imaginary parts. 

\begin{lemma} \label{supportCtoR}
	Let $M$ be a complete pointed metric space and let $\gamma \in \F(M,\C)$. Then $\supp(\gamma) = \supp(\re(\gamma)) \cup \supp(\im(\gamma))$. 
\end{lemma}

\begin{proof}
	If $x \in M$ then $x \not\in \supp(\gamma)$ if and only if there exists a neighbourhood $U$ of $x$ such that for every $f \in \Lip_0(M,\R)$ whose support is contained in $U$, we have $\langle f , \gamma \rangle = 0$. Since $\gamma = \re(\gamma) + i\, \im(\gamma)$ and $\langle f ,  \re(\gamma) \rangle , \langle f , \im(\gamma) \rangle \in \R$, $\langle f , \gamma \rangle = 0$ if and only if $\langle f ,  \re(\gamma) \rangle = 0$ and $\langle f , \im(\gamma) \rangle = 0$. Hence $x \not\in \supp(\gamma)$ if and only if $x \not\in \supp(\re(\gamma))$ and $x \not\in \supp(\im(\gamma))$.
\end{proof}

In the next lemma, which is simply the complex version of Lemma 2.10 in \cite{ACP20}, recall that $|\supp \gamma |$ stands for the cardinality of the support of $\gamma$.

\begin{lemma} \label{WeaklyClosed} Let $M$ be a complete pointed metric space.
Let $\FS_n(M,\C) = \{\gamma \in \F(M,\C) \; : \; |\supp \gamma | \leq n  \}$. Then $\FS_n(M,\C)$ is weakly closed in $\F(M,\C)$.
\end{lemma}

\begin{proof}
Aiming for a contradiction, suppose $(\gamma_i)_i \subset \FS_n(M,\C)$ is a net which weakly converges to some $\gamma \not\in \FS_n(M,\C)$. This means that $\supp(\gamma)$ contains at least $n+1$ points $x_1,\ldots,x_{n+1}$. Let $\delta>0$ be small enough so that the balls $B(x_k,\delta)$, for $k=1,\ldots,n+1$, are pairwise disjoint. By Proposition~\ref{CharacSupport}, there are $f_k\in\Lip_0(M)$ such that $\supp(f_k)\subset B(x_k,\delta)$ and $\langle f_k , \gamma \rangle \neq 0$. Therefore, if $i$ is large enough we must have $\langle f_k , \gamma_i \rangle \neq 0$ for every $k$, hence $\supp(\gamma_i)\cap B(x_k,\delta)\neq\emptyset$ for every $k$. This is impossible since $\supp(\gamma_i)$ only has $n$ elements.
\end{proof}

Finally the following important result is simply the complex version of \cite[Theorem~C]{ACP21}.

\begin{theorem} \label{ThmNormToWeak}
	Let $M$ be a pointed complete metric space. If a sequence $(\gamma_n)_n \subset \FS_n(M , \C)$
	weakly converges to some $\gamma  \in \F(M, \C)$, then $\gamma \in \FS_n(M , \C)$ and $(\gamma_n)_n$ actually converges to $\gamma$ in the norm topology.
\end{theorem}

\begin{proof}
	Thanks to Lemma~\ref{supportCtoR}, $| \supp(\re(\gamma_n)) \cup \supp(\im(\gamma_n))| \leq n$. It is easily checked that $\overline{\gamma_n} \to \overline{\gamma}$ in the weak topology, and so $\re(\gamma_n) \to \re(\gamma)$ and $\im(\gamma_n) \to \im(\gamma)$ in the weak topology. Hence, by \cite[Theorem~C]{ACP21}, $\re(\gamma_n) \to \re(\gamma)$ and $\im(\gamma_n) \to \im(\gamma)$ in the norm topology, with moreover $| \supp(\re(\gamma)) \cup \supp(\im(\gamma)) | \leq n$. Applying Lemma~\ref{supportCtoR} again yields that $\gamma \in \FS_n(M , \C)$ and $(\gamma_n)_n$ actually converges to $\gamma$ in the norm topology.
\end{proof}

\textbf{From now on, unless otherwise specified, we tacitly assume that $\K=\C$ and therefore we will omit referring to $\K$. Of course, all the results in the sequel hold true in the case of real scalars.}

\subsection{A last remark about completeness}\label{remarkcomplete}

Let $M$ be any pointed metric space and let $\overline{M}$ be its completion. It is quite standard that any Lipschitz map $f : M \to \C$ admits a unique extension to $\overline{M}$ with the same Lipschitz constant $\|f\|_L$. On the other hand, the restriction to $M$ of a Lipschitz function $g : \overline{M} \to \C$ also has the same Lipschitz constant as the original map.  These two easy observations yield that the spaces $\Lip_0(M)$ and $\Lip_0(\overline{M})$ are linearly isometric as Banach spaces. The same thing can be said about their predual, that is about the spaces $\F(M)$ and $\F(\overline{M})$. At various places, we will require the spaces to be complete for technical issues (mostly when the notion of support comes into play). But, thanks to these last isometric identifications, no completeness assumptions are necessary for most of our theorems concerning weighted operators.

\section{Boundedness, injectivity and surjectivity} \label{bounded}

In this section, $M,N$ are pointed metric spaces and $f : M \to N$, $w:M \to \C$ are maps such that either $f(0_M)=0_N$ or $w(0_M)=0$. Recall that the \textit{weighted composition operator} $wC_f : \Lip_0(N) \mapsto \C^M$ is defined by:$$\forall g \in \Lip_0(N), \;  \forall x \in M, \quad wC_f(g)(x)  = w(x) \cdot g ( f(x) ).$$
Also remind that the \textit{weighted Lipschitz operators} is the unique extension, when it exists, of the linear map $w \widehat{f} : \lspan \delta(M) \to \lspan \delta(N)$ such that:
$$ \forall \gamma = \sum_{i=1}^n a_i \delta(x_i) \in  \F(M), \quad  w \widehat{f}(\gamma) = \sum_{i=1}^n w(x_i) \delta(f(x_i)).$$

\subsection{Boundedness}

In order to characterize bounded weighted Lipschitz operators, we will need the next result which should be compared with \cite[Lemma~11]{Cuth}.

\begin{lemma}\label{lemma:formula}
	Le $M$ be a pointed metric space. Let $\lambda_1,\lambda_2 \in \C$ and $x,y \in M$. 
	If 
	\begin{align*}
		M:=\max \big\{|\lambda_1 d(x,0) + \lambda_2 d(y,0)|,\; 
	 |\lambda_1d(x,0) +\lambda_2 (d(x,0)-d(x,y))|,\qquad \qquad \\
	 \hfill | \lambda_2 d(y,0) + \lambda_1(d(y,0)-d(x,y))| \big\}
	\end{align*}
	then 
	$$\frac{1}{\sqrt{2}}M \leq  \|\lambda_1 \delta(x) + \lambda_2 \delta(y)\|_{\F(M)} \leq \sqrt{2}M. $$
\end{lemma}

\begin{proof} Let us write $\gamma = \lambda_1 \delta(x) + \lambda_2 \delta(y)$. 
	It follows from Section~\ref{sectionFC} that
	$\|\re(\gamma)\|_{\F(M,\C)} = \|\re(\gamma)\|_{\F(M,\R)}$, $\|\im(\gamma)\|_{\F(M,\C)} = \|\im(\gamma)\|_{\F(M,\R)}$ and
	$$\max\{\|\re(\gamma)\| , \|\im(\gamma)\|\} \leq \|\gamma\|_{\F(M,\C)} \leq \|\re(\gamma)\| + \|\im(\gamma)\|.$$
	By \cite[Lemma~11]{Cuth}, if $a,b \in \R$ then 
	\begin{align*}
		\|a \delta(x) + b \delta(y)\|_{\F(M,\R)} 
		=   \max \big\{|ad(x,0) + bd(y,0)|,\; 
		& |ad(x,0) + b(d(x,0)-d(x,y))|,\; \\
		& |bd(y,0) + a(d(y,0)-d(x,y))| \big\}.
	\end{align*}
	Now the conclusion follows from the basic inequalities $\|u\|_2 \leq \sqrt{2}\|u\|_\infty$ and $\|u\|_1 \leq \sqrt{2}\|u\|_2$ in $\R^2$.
\end{proof}

The second lemma is an adaptation of \cite[Theorem~2.3]{Vargas2} for $\K = \C$.
In the statement below, we let 
$$\mathcal{M} := \left\{ d(x,y)^{-1}(\delta(x) - \delta(y)) \; : \; x \neq y \in M \right\} \subset \F(M)$$ 
to be the set of elements generally called \textit{molecules}. 

\begin{lemma}\label{lemma:AbsConvHull}
	Let $M$ be a pointed metric space. Then $B_{\F(M)} = \overline{\aco}  \mathcal{M}$. 
\end{lemma}

\begin{proof}
	The inclusion $\overline{\aco} \mathcal{M} \subset B_{\F(M)}$ is obvious. Conversely, assume that there exists $\gamma \in B_{\F(M)} \setminus \overline{\aco} \mathcal{M}$. By the Hahn--Banach separation theorem, there exists $f \in S_{\Lip_0(M)}$ such that 
	$$\re(\langle f , \gamma \rangle )  > \sup \{\re(\langle f,\mu\rangle)  \; : \; \mu \in \overline{\aco} \mathcal{M} \}).$$
	For $\mu \in \overline{\aco} \mathcal{M}$, write $\langle f,\mu\rangle = re^{i\theta}$. Note that $e^{-i\theta}\mu \in \overline{\aco} \mathcal{M}$ and
	$\langle f, e^{-i\theta}\mu\rangle = r = |\langle f,\mu\rangle|.$ 
	Hence 
	$$
	\sup \{\re(\langle f,\mu\rangle)  \; : \; \mu \in \overline{\aco} \mathcal{M} \}) = \sup \{|\langle f,\mu\rangle|  \; : \; \mu \in \overline{\aco} \mathcal{M} \}) =\|f\| = 1.
	$$
	Finally,
	$
	1 = \|f\| \geq \re(\langle f , \gamma \rangle ) > 1,
	$
	which is a contradiction.
\end{proof}

Before stating the main result of this section, let us introduce some notation. For $x,y \in M$ such that $x \neq y$, we consider
\begin{align*}
	A(x,y) &= \frac{1}{d(x,y)}|w(x)d(f(x),0) - w(y)d(f(y),0)|,  \\
	B(x,y) &= \frac{1}{d(x,y)}|w(x)d(f(x),0) - w(y)(d(f(x),0) - d(f(x),f(y))|.
\end{align*}
Following the characterization obtained in \cite[Theorem~2.1]{GolMa22}, we also introduce 
\begin{align*}
	\sigma(x,y) &= \frac{d(f(x),f(y))}{d(x,y)}(s(x,y)|w(x)|+s(y,x)|w(y)|),\\
	\tau(x,y)   &= \frac{|w(x)-w(y)|}{d(x,y)} \min\{d(f(x),0) , d(f(y),0) \}, 
\end{align*}
where $s(x,y) = 1$ if $d(f(x),0)\geq d(f(y),0)$, and $s(x,y) = 0$ otherwise.

\begin{theorem} \label{Poids}
	Let $M,N$ be pointed metric spaces. Let $w : M \to \C$ and $f : M \to N$ be any maps such that $f(0_M)=0_N$ or $w(0_M)=0$. 
	The following assertions are equivalent:
	\begin{enumerate}[itemsep = 4pt]
		\item[$(i)$]  $w \widehat{f}$ extends to a bounded operator from  $\F(M)$ to $\F(N)$;
		\item[$(ii)$] $wC_f$ defines a bounded operator from $\Lip_0(N)$ to $\Lip_0(M)$ (and $wC_f = (w \widehat{f})^*$);
        \item[$(iii)$] For every $g \in \Lip_0(N)$, $wC_f(g) \in \Lip_0(M)$.
		\item[$(iv)$] $\varphi : x \in M \mapsto w(x)\delta(f(x)) \in \F(N)$ is Lipschitz (and $\overline{\varphi} = w \widehat{f}$);
		\item[$(v)$] $A:= \displaystyle \sup_{x \neq y} A(x,y) < \infty$ and $B := \displaystyle \sup_{x \neq y} B(x,y) < \infty$;
		\item[$(vi)$]  $\sigma:= \displaystyle \sup_{x \neq y} \sigma(x,y) < \infty$ and $\tau:= \displaystyle \sup_{x \neq y} \tau(x,y) < \infty$.
	\end{enumerate}
In that case, one has $\frac{1}{\sqrt{2}}\max(A,B) \leq \|wC_f\| = \|w\widehat{f} \| \leq \sqrt{2}\max(A,B).$
\end{theorem}

It is worth mentioning that in the case $\K = \R$ one has $\|wC_f\| = \|w\widehat{f} \| = \max(A,B).$

\begin{proof}
	$(i) \implies (ii)$. If $(w\widehat{f})^*: \Lip_0(N) \to \Lip_0(M)$ is the adjoint of $w\widehat{f}$ then:
	\begin{align*}
		\forall g \in \Lip_0(N), \; \forall x \in M, \quad (w \widehat{f})^*(g)(x) &= \langle (w \widehat{f})^*(g) , \delta(x) \rangle = \langle g , w \widehat{f}(\delta(x)) \rangle \\
		&= \langle g , w(x) \delta(f(x)) \rangle = w(x) g(f(x)) \\
		&= wC_f(g)(x).
	\end{align*}
	Therefore $(w\widehat{f})^* = wC_f$ and so $wC_f$ is bounded.
	\medskip
	
	$(ii) \implies (i)$. We may proceed similarly. Indeed, it is immediate from the definitions that $wC_f(g_i) \to wC_f(g)$ pointwise whenever $g_i \to g$ pointwise in $B_{Lip_0(N)}$. Thus $wC_f$ is continuous for the topology of pointwise convergence on bounded subsets of $\Lip_0(N)$. So it is weak$^*$-to-weak$^*$ continuous by the Banach-Dieudonn\'e theorem. As in the previous implications, it is readily seen that the pre-adjoint operator of $wC_f$ is equal to $w\widehat{f}$. 
	\medskip

    $(ii) \iff (iii)$ is a straightforward application of the closed graph theorem. 
	\medskip 
 
	$(i) \iff (iv)$. By Lemma \ref{lemma:AbsConvHull}, $B_{\F(M)} = \overline{\aco} \mathcal{M}$ where $\mathcal{M}$ is the set of molecules. Hence, $w \widehat{f}$ is bounded if and only if it is uniformly bounded on molecules.  That is, $w \widehat{f}$ is bounded if and only if $\sup_{x\neq y} \|w \widehat{f} (m_{x y})\|_{\F(N)} < \infty$. Now the desired equivalence follows the next observation
	$$
	\|w \widehat{f} (m_{x y})\|= \frac{\|w(x) \delta(f(x)) - w(y) \delta(f(y))\|}{d(x,y)} = \frac{\|\varphi(x) - \varphi(y)\|}{d(x,y)}.
	$$
    Moreover, if $\varphi : x \in M \mapsto w(x)\delta(f(x)) \in \F(N)$ is Lipschitz, then its unique extension to $\F(M)$ using Theorem~\ref{thm1} verifies $\overline{\varphi} = w \widehat{f}$.
	\medskip
	
	To conclude $(iv) \iff (v)$ is an easy consequence of Lemma~\ref{lemma:formula} and $(v) \iff (vi)$ follows from the next lemma.
\end{proof}

\begin{lemma} \label{inequalities}
	For every $x \neq y \in M$, 
	\begin{enumerate}[$(i)$]
		\item $\sigma(x , y) \leq 2\big(A(x , y)+\max\big(B(x , y) , B(y, x)\big)\big)$;
		\item $ \tau(x ,y) \leq A(x , y) + \sigma(x ,y)$;
		\item $A(x,y) \leq \sigma(x,y) + \tau(x,y)$
		\item $\max\big(B(x , y) , B(y, x)\big) \leq A(x,y) + 2 \sigma(x,y)$
	\end{enumerate}
\end{lemma}

\begin{proof}
Let $x\neq y \in M$. Without loss of generality, we may assume that $d(f(x) , 0) \leq d(f(y),0)$.
If $d(f(x) , 0) < d(f(y),0)$ then 
\begin{align*}
	\sigma(x,y) &= |w(y)| \dfrac{d(f(x),f(y))}{d(x,y)} < |w(y)| \dfrac{d(f(y),0) - d(f(x),0) + d(f(x),f(y))}{d(x,y)}.
\end{align*}
Now write
\begin{align*}
	& w(y) \dfrac{d(f(y),0) - d(f(x),0) + d(f(x),f(y))}{d(x,y)} \\
	& \underset{(\star)}{=} \dfrac{w(x)d(f(x),0) - w(y)(d(f(x),0)-d(f(x),f(y)))}{d(x,y)} -\dfrac{w(x)d(f(x),0) - w(y)(d(f(y),0)}{d(x,y)}. 
\end{align*}
It is then clear that $\sigma(x,y) \leq A(x,y) + B(x,y) \leq A(x,y) + \max(B(x,y) , B(y,x))$. If $d(f(x) , 0) = d(f(y),0)$ then we obtain $\sigma(x,y) \leq 2\big(A(x , y)+\max\big(B(x , y) , B(y, x)\big)\big)$ in a similar way.
\smallskip

We handle $\tau(x,y)$ as follows:
\begin{align*}
	\tau(x,y)  &= \frac{|w(x)-w(y)|}{d(x,y)}d(f(x),0) \\
	&\leq \frac{|w(x)d(f(x),0)-w(y)d(f(y),0)|}{d(x,y)} + |w(y)|\frac{|d(f(y),0) - d(f(x),0)|}{d(x,y)} \\
	&\leq A(x,y) + |w(y)|\frac{|d(f(x),f(y))|}{d(x,y)} \leq A(x,y) + \sigma(x,y).
\end{align*}
Next, the triangle inequality yields:
\begin{align*}
	A(x,y) &\leq d(f(x),0) \frac{|w(x)- w(y)|}{d(x,y)} + |w(y)|\frac{|d(f(x),0) - d(f(y),0)|}{d(x,y)} \\
		& \leq \tau(x,y) + \sigma(x,y).
\end{align*}
Finally, the equality $(\star)$ implies that $B(x,y) \leq A(x,y) + 2\sigma(x,y)$ and the inequality for $B(y,x)$ is obtained in the same way. 
\end{proof}

\begin{remark} \label{remarkbounded}~
	\begin{enumerate}[leftmargin = *, itemsep=5pt]
     \item If $w \equiv 1$ then $(iv)$ and $(v)$ in Theorem~\ref{Poids} are equivalent to $f$ being simply a Lipschitz map. Nevertheless $f$ needs not be Lipschitz in general; see e.g. the example below. 
	\item It is clear that the conditions in $(iv)$ and $(v)$ in Theorem~\ref{Poids} are implied by 
		$$N_1:=\displaystyle \sup_{x \neq y} |w(x)| \frac{d(f(x) , f(y))}{d(x,y)} < \infty$$
		and 
		$$N_2:= \displaystyle  \sup_{x \neq y} d(f(x) ,0 ) \frac{|w(x) - w(y)|}{d(x,y)}< \infty .$$
		In fact, if $(iv)$ (or $(v)$) is satisfied, then $N_1 < \infty$ if and only if $N_2 < \infty$. However, the two latter conditions together are strictly stronger than $(v)$.
		For instance, consider $M = N = \{0\} \cup [1,+\infty[$ with the metric inherited from $\R$. Let $f(x) = x^2$ while $w(x) = \frac{1}{x}$ if $x \neq 0$, $w(0)=0$ otherwise. Now observe that
		$$|w(1)|\frac{d(f(1),f(n))}{d(1,n)} = \frac{|n^2-1|}{|n-1|} \underset{n \to +\infty}{\longrightarrow} +\infty \implies N_1 = +\infty. $$
		Nevertheless, if $0<x < y$ then 
		\begin{align*}
			|w(y)|\frac{d(f(x),f(y))}{d(x,y)} &= \frac{1}{y} \frac{y^2-x^2}{y-x} = \frac{y+x}{y} \leq \frac{2y}{y} = 2, \\
			|w(y)|\frac{d(f(0),f(y))}{d(0,y)} &= \frac{1}{y} \frac{y^2}{y} \leq 1.
		\end{align*}
		We deduce that $\sigma \leq 2$. Lastly:
		\begin{align*}
			|d(f(x),0)|\frac{|w(x)-w(y)|}{d(x,y)} &= x^2 \frac{\frac{1}{x}-\frac{1}{y}}{y-x} = \frac{x^2}{xy} \leq \frac{x}{y} \leq 1 \implies \tau \leq 1.
		\end{align*}
	\end{enumerate}
\end{remark}
\medskip

Using Section~\ref{Section1.2}, we deduce the next simpler statement for the Lip variant of Theorem~\ref{Poids}. It improves \cite[Theorem~2.1]{GolMa16} where only compact metric spaces are considered. 

\begin{corollary}
Let $M,N$ be  metric spaces. Let $w : M \to \C$ and $f : M \to N$ be any maps.    Consider the composition operator $wC_f : \Lip(N) \to \C^M$ given by $$\forall g \in \Lip(N), \;  \forall x \in M, \quad wC_f(g)(x)  = w(x) \cdot g ( f(x) ).$$
    Then $wC_f$ is a bounded operator from $\Lip(N)$ to $\Lip(M)$ if and only if $w$ is bounded, Lipschitz and 
    $$N_1:=\sup_{x \neq y} |w(x)| \frac{d(f(x) , f(y))}{d(x,y)} < \infty.$$
\end{corollary}

\begin{proof}
    According to Proposition~\ref{Lip0toLip}, $wC_f$ is a bounded operator from $\Lip(N)$ to $\Lip(M)$ if and only if $w^eC_{f^e}$ is a bounded operator from $\Lip_0(N^e)$ to $\Lip_0(M^e)$. Thanks to Theorem~\ref{Poids}, $w^eC_{f^e}$ is bounded if and only if 
    \begin{align*}
	\sigma = \sup_{x\neq y} \sigma(x,y) &= \sup_{x\neq y}\frac{d^e(f^e(x),f^e(y))}{d^e(x,y)}(s(x,y)|w^e(x)|+s(y,x)|w^e(y)|) < \infty, \text{ and }\\
	\tau = \sup_{x\neq y} \tau(x,y)   &= \sup_{x\neq y} \frac{|w^e(x)-w^e(y)|}{d^e(x,y)} \min\{d^e(f^e(x),e) , d^e(f^e(y),e) \} < \infty.
\end{align*}
    Since $d(z,e)=1$ for every $z\in N$, we get $s(x,y)=s(y,x)=1$ for every $x\neq y \in M$. Thus it follows fairly easily that $\sigma < \infty$ if and only if $N_1 < \infty$. Besides, taking the next observations into account, one can finish the proof in an obvious way. First, note that letting $y=e$ in $\sigma(x,y)$ gives
    $$ \sigma(x,e) =  \frac{d^e(f^e(x),f^e(e))}{d^e(x,e)}|w^e(x)| = |w(x)|.$$
    Thus $w$ is bounded in any case. Second, if $x \neq y \in M$ then:
    \begin{align*}
        d(x,y) \leq 2 \implies d^e(x,y) = d(x,y) \implies \tau(x,y) = \frac{|w(x) - w(y)|}{d(x,y)};\\
        d(x,y) > 2 \implies d^e(x,y) = 2 \implies \tau(x,y) = \frac{|w(x) - w(y)|}{2} \leq \|w\|_{\infty}.
    \end{align*}
\end{proof}

Finally, the next corollary is \cite[Theorem~2.1]{DA}, it is a direct consequence of Theorem~\ref{Poids} so we leave the details to the reader.  

\begin{corollary}
    Let $M$ be a bounded metric space. If $w$ is Lipschitz and 
		$$\sup_{x \neq y} |w(x)| \frac{d(f(x) , f(y))}{d(x,y)} < \infty$$
  then $wC_f$ is a bounded operator from $\Lip_0(N)$ to $\Lip_0(M)$.
\end{corollary}

\begin{remark}
    When $M$ is bounded, $w$ needs not be Lipschitz nor bounded for $wC_f$ to be bounded from $\Lip_0(N)$ to $\Lip_0(M)$; an example can be found in \cite{DA} (see Example~2 therein). But, in this case, $\Lip_0(M)$ can be seen as 1-codimensional subspace of $\Lip(M)$ (isomorphically speaking). Precisely, $\Lip(M)= \Lip_0(M) \oplus \lspan (\mathbbm{1})$, where $\mathbbm{1}$ denotes the function constantly equal to 1. Assume now that $M$ and $N$ are both bounded and that $wC_f$ is bounded from $\Lip_0(N)$ to $\Lip_0(M)$. Then $wC_f$ can be extended to become a bounded operator from $\Lip(N)$ to $\Lip(M)$ if and only if $wC_f(\mathbbm{1}) (=w) \in \Lip(M)$.
\end{remark}

\subsection{Injectivity of weighted composition operators}

Recall that a bounded linear operator $T:X\to Y$ between Banach spaces has a dense range if and only if its adjoint $T^* : Y^* \to X^*$ is one-to-one. The following proposition corresponds to \cite[Theorem 3.1]{DA}, with a proof only using the tools of Lipschitz free spaces.

\begin{proposition} \label{injectivity} 	Let $M,N$ be pointed metric spaces. 
A bounded weighted composition operator $\omega C_f:\Lip_0(N)\to \Lip_0(M)$ is injective if and only if $f(\text{coz}(w)) \cup\{0\})$ is dense in $N$.
\end{proposition}

\begin{proof}
First, following the proof of Proposition 2.1 in \cite{ACP20}, notice that it is effortless to check that if $L \subset N$, then $\F(L) := \overline{\text{span}}(\delta(L))$ is dense in $\F(N)$ if and only if $L$ is dense in $N \setminus \{ 0 \}$, that is, if and only if $L \cup \{ 0 \}$ is dense in $N$. Next, we have
\begin{align*}
wC_f:\Lip_0(N)\to \Lip_0(M) \ \text{is one-to-one}
& \Longleftrightarrow
 (w\widehat{f})(\F(M)) \ \text{is dense in} \ \F(N) \\
 & \Longleftrightarrow
 (w\widehat{f})(\text{span} ~ \delta(M)) \ \text{is dense in} \ \F(N).
\end{align*}
By definition of $\text{coz}(w)$, we have $(w\widehat{f})(\text{span}(\delta(M))) = \text{span} ~ \delta(f(\text{coz}(w)))$, from which we get
\begin{align*}
wC_f:\Lip_0(N)\to \Lip_0(M) \ \text{is one-to-one}
& \Longleftrightarrow \F(f(\text{coz}(w))) \ \text{is dense in} \ \F(N) \\
 & \Longleftrightarrow f(\text{coz}(w)) \cup \{ 0 \} \ \text{is dense  in} \ N,
\end{align*}
and this concludes the proof.
\end{proof}

\subsection{Surjectivity of weighted composition operators}

Finally, we characterize the surjectivity of $wC_f$ for general metric spaces. 
The following proposition should be compared with \cite[Theorem 3.5]{GolMa16} where a characterization is provided when the space $M$ is compact. Note that the characterization below is less easy to check than the one in the aforementioned theorem. However, for a general metric space $M$, it seems that there was no known characterization. For partial results, we refer to \cite[Theorem 3.2]{DA} where a sufficient condition is given while necessary conditions are provided in \cite[Theorem 3.4]{DA}.

\begin{proposition} \label{surjectivity}
Let $M,N$ be pointed metric spaces. Let $w : M \to \C$ and $f : M \to N$ be a map such that $f(0)=0$ or $w(0)=0$. Assume that $wC_f:\Lip_0(N)\to \Lip_0(M)$ is bounded. Then, the following assertions are equivalent:
\begin{enumerate}[itemsep = 4pt]
\item[$(i)$] $\omega C_f:\Lip_0(N)\to \Lip_0(M)$ is surjective.
\item[$(ii)$] $w$ does not vanish on $M\setminus \{ 0 \}$, $f$ is injective and, writing $\widetilde{w} : f(M) \to \C$ the weight defined by $\widetilde{w}(z) = \dfrac{1}{w(f^{-1}(z))}$ if $z \neq 0$ and $\widetilde{w}(0)=0$, the operator $\widetilde{w} \widehat{f^{-1}} : \F(f(M)) \to \F(M)$ is bounded.
\item[$(iii)$] $w$ does not vanish on $M\setminus \{ 0 \}$, $f$ is injective,
$$
\sup_{x\neq y}  \ \frac{1}{d(f(x),f(y))} \left| \dfrac{d(x,0)}{w(x)} - \dfrac{d(y,0)}{w(y)} \right| < +\infty, \text{ and}
$$
$$
\sup_{x\neq y} \ \frac{1}{d(f(x),f(y))} \left| \dfrac{d(x,0)}{w(x)} - \dfrac{d(x,0) - d(x,y)}{w(y)}  \right| < +\infty,
$$
with the convention $\dfrac{1}{w(x)} = 0$ if $x=0$.
\end{enumerate}
\end{proposition}

\begin{proof} 
$(i) \implies (ii)$. First, it is clear that if $w\widehat{f}$ is injective then $w$ does not vanish on $M\setminus \{ 0  \}$ and $f$ is injective. Indeed, if $w(x) = 0$ or $f(x)=0$ with $x\neq 0$, then $w\widehat{f}(\delta(x)) = 0$. Also if $f(y) = f(z)$ with $y\neq z$, then $w\widehat{f}\left(\delta(y) - \frac{w(y)}{w(z)} \delta(z) \right) = 0$  while $\delta(y) - \frac{w(y)}{w(z)} \delta(z) \neq 0$. Next, as recalled in the introduction, $wC_f$ is surjective if and only if its preadjoint $w\widehat{f}$ is injective and has a closed range, that is, if $w\widehat{f} : \F(M) \to  \F_N(f(M)) \simeq \F(f(M))$ is an isomorphism. When $w$ does not vanish on $M\setminus \{ 0  \}$ and $f$ is injective, the inverse of $T := w\widehat{f} : \F(M) \to \F(f(M)) $ 
is given, for every $y \in f(M)\setminus \{ 0_N \}$, by
$$
T^{-1}(\delta(y)) = \dfrac{1}{w(f^{-1}(y))} \delta(f^{-1}(y)) = \widetilde{w}(y) \delta(f^{-1}(y)).
$$
Thus $T^{-1} = \widetilde{w}\widehat{f^{-1}}$, and this yields $(ii)$.\\
$(ii) \implies (i)$.
Simply note that $w\widehat{f} : \F(M) \to \F(f(M))$ is an isomorphism with inverse mapping being $\widetilde{w}\widehat{f^{-1}}$, therefore $wC_f$ is surjective.\\
$(ii) \Longleftrightarrow (iii)$. It is a consequence of Theorem \ref{Poids} applied to $\widetilde{w}\widehat{f^{-1}}$.
\end{proof}

\section{Compact and weakly compact weighted operators}\label{sectioncompact}

Pursuing our study of weighted operators, we will now focus on compactness and weak compactness of operators of the kind  $w \widehat{f} : \F(M) \to \F(N)$ and $wC_f: \Lip_0(N) \to \Lip_0(M)$. The first result which we present below will be very useful in the sequel. It is the straightforward extension of \cite[Theorem~2.3]{Vargas2}.

\begin{proposition} \label{caracCompactWeight}
	Let $M,N$ be pointed metric spaces.  Let $w : M \to \C$ and $f : M \to N$ be any maps such that $f(0)=0$ or $w(0)=0$. If $w\widehat{f} : \F(M) \to \F(N)$ is a bounded operator then the following are equivalent: 
	\begin{enumerate}[$(i)$]
		\item $w\widehat{f}$ is (weakly) compact;
		\item $wC_f = (w\widehat{f})^*$ is (weakly) compact;
		\item $\Big\{ \dfrac{w(x)\delta(f(x)) - w(y)\delta(f(y))}{d(x,y)} \, : \, x \neq y \in M \Big\}$
		is relatively (weakly) compact.
	\end{enumerate}
\end{proposition}

\begin{proof}
	The equivalence $(i) \iff (ii)$ follows from Schauder's theorem for norm compactness (see e.g. \cite[Theorem~3.4.15]{Megg}) and from Gantmacher's theorem for weak compactness (see e.g. \cite[Theorem~3.5.13]{Megg}). 
	For the direction $(i) \implies (iii)$, notice that 
	$$ \left\{ \dfrac{w(x)\delta(f(x)) - w(y)\delta(f(y))}{d(x,y)} \; : \; x \neq y \in M \right\} = w\widehat{f}(\mathcal{M}),$$
	where $\mathcal{M} = \left\{ d(x,y)^{-1}(\delta(x) - \delta(y)) \; : \; x \neq y \in M \right\}$. Since $\mathcal{M} \subset B_{\F(M)}$, if $w\widehat{f}$ is (weakly) compact then $w\widehat{f}(\mathcal{M})$ must be relatively (weakly) compact. Let us finish the proof by showing that $(iii) \implies (i)$. Thanks to Lemma~\ref{lemma:AbsConvHull}, we have $B_{\F(M)} = \overline{\aco} \mathcal{M}$. Now observe that by boundedness of $w\widehat{f}$
	$$ w\widehat{f}(B_{\F(M)}) \subset  w\widehat{f}(\overline{\aco}\mathcal{M}) \subset \overline{\aco} (w\widehat{f}(\mathcal{M})) \subset \overline{\aco} \left(\overline{w\widehat{f}(\mathcal{M})}\right).$$
	So, if $w\widehat{f}(\mathcal{M})$ is relatively (weakly) compact, then $ \overline{\aco} \left(\overline{w\widehat{f}(\mathcal{M})}\right)$ is (weakly) compact (see e.g. \cite[Corollary~5.31 and Theorem~5.35]{IDBook} for norm compactness and \cite[Theorem~10.15]{KreinTHM} for weak compactness). Hence  $\widehat{f}(B_{\F(M)})$ is relatively (weakly) compact. 
\end{proof}

Next, we wish to highlight that $w \widehat{f}$ and $wC_f$ are in fact compact if and only if they are weakly compact. This was already known in the real case when $w \equiv 1$, see \cite[Theorem~B]{ACP21}. We use the same arguments to derive the case of weighted operators. 

\begin{corollary}\label{compactiffweakly}
	Let $M,N$ be complete pointed metric spaces.  Let $w : M \to \C$ and $f : M \to N$ be any maps such that $f(0)=0$ or $w(0)=0$ and $w\widehat{f}$ is bounded. Then $w \widehat{f}$ and $wC_f$ are compact if and only if they are weakly compact.
\end{corollary}

\begin{proof}
	Thanks to Theorem~\ref{ThmNormToWeak} and the Eberlein--\v{S}mulian theorem, 
	the set
	$$\Big\{ \dfrac{w(x)\delta(f(x)) - w(y)\delta(f(y))}{d(x,y)} \, : \, x \neq y \in M \Big\} \subset \FS_2(N)$$ 
	is  relatively weakly compact if and only if it compact. The conclusion now follows from Proposition~\ref{caracCompactWeight}.
\end{proof}

We will now investigate some metric conditions on $w$ and $f$ under which the considered operators are compact. For this matter, Proposition~\ref{caracCompactWeight} shows that it is crucial to have a better understanding of sequences of finitely supported elements in Lipschitz free spaces. This is exactly the purpose of the next lemmas.

\begin{lemma}\label{PrepmaintheoremNew}
	Let $M$ be a complete pointed metric space. Let $k\in \mathbb{N}$ and $(\gamma_n)_n \subset \mathcal{FS}_k(M)$ be a sequence converging weakly to an element $\gamma \in \mathcal{FS}_k(M)$. Then, for every $p\in \supp(\gamma)$, there exists $(x_n)_n \subset M$ such that  $\lim\limits_{n \to + \infty}\,d(x_n,p)=0$ and $x_n \in \supp(\gamma_n)$ for every $n \in \N$.
\end{lemma} 

Before going into the proof, it is worth mentioning that the base point cannot be an isolated point of $\supp(\gamma)$ for $\gamma \in \F(M)$. Therefore $0$ does not belong to the support of any element in $\mathcal{FS}_k(M)$.

\begin{proof}
	Aiming for a contradiction, assume that for every sequence $(x_n)_n$ with $x_n \in \supp(\gamma_n)$, there exists $\ep >0$ such that $\limsup_n d(x_n,p)>\ep$. For every $n \in \N$, we pick $z_n \in \supp(\gamma_n)$ such that $d(z_n , p) = \min \{ d(z,p) \, | \; z \in \supp(\gamma_n) \}$ (the minimum exists since $|\supp(\gamma_n)| \leq k$). Let us fix $\ep >0$ such that $d(p , \supp(\gamma) \setminus \{p\}) > \ep$ and $\lim_k d(z_{n_k},p) > \ep$ for some well chosen subsequence $(z_{n_k})_k$. By Lemma~\ref{LemmaConstrctionLipF}, we can find $h \in \Lip_0(M,\R)$ such that $h(p) = 1$, and $h \equiv 0$ outside of $B(p, \ep/2)$.
	Now notice that $\< h , \gamma_{n_k} \> =0$ for every $k \in \N$ while $\lim_k \< h , \gamma_{n_k} \> = \< h , \gamma \> =1$ since $\gamma_{n} \to \gamma$ in the weak topology. This contradiction concludes the proof.
\end{proof}

\begin{lemma} \label{SequencesInFS2}
	Let $M$ be a complete pointed metric space. Let $(\gamma_n)_n \subset \mathcal{FS}_2(M)$ be a sequence which weakly converges to some $\gamma \in \F(M)$. Then exactly one of the following cases occurs:
	\begin{enumerate}[$(i)$, itemsep=5pt]
		\item $|\supp(\gamma)| = 2$, $\gamma =a\delta(p) + b \delta(q)$, and we may write $\gamma_n = a_n \delta(x_n) + b_n \delta(y_n)$ with $a_n \to a$, $b_n \to b$, $x_n \to p$ and $y_n \to q$.
		\item $|\supp(\gamma)| = 1$, $\gamma =a\delta(p)$, and we may write  $\gamma_n = a_n \delta(x_n) + b_n \delta(y_n)$ with  $x_n \to p$ and \\
		- either $y_n \to p$, $a_n+b_n \to a$, $|a_n|d(x_n,y_n) \to 0$ and $|b_n|d(x_n,y_n) \to 0$;\\
		- or $a_{n_k} \to a$ and $b_{n_k}d(y_{n_k},0) \to 0$ for some increasing subsequence $(n_k)_k \subset \N$.
		\item $\gamma = 0$  and we may write $\gamma_n = a_n \delta(x_n) + b_n \delta(y_n)$ with $a_n d(x_n,0) + b_n d(y_n,0)\to 0$,
	 $a_n d(x_n,0) +b_n (d(x_n,0)-d(x_n,y_n))\to 0$,
  $b_n d(y_n,0) + a_n(d(y_n,0)-d(x_n,y_n)) \to 0$.
	\end{enumerate}
\end{lemma}

\begin{remark} \label{converse}
\begin{enumerate}[leftmargin=* , itemsep=5pt]
    \item In each one of the cases presented above, a converse holds in the following sense:
    if $(\gamma_n)_n \subset \mathcal{FS}_2(M)$ is a sequence verifying one of the description in $(i)$ or $(iii)$, then $(\gamma_n)_n$ converges to $\gamma$ described in the corresponding case. The same statement holds for the first case in $(ii)$, while the second case in $(ii)$ yields the convergence of a subsequence of $(\gamma_n)_n$. The details are easy and left to the reader. 
    \item Note that in the case $(i)$ of the previous Lemma, the fact that the limit has two elements in its support forces the elements $\gamma_n$ (for $n$ large enough) to have two elements in their support as well. But since the family $(\delta(x))_{x\in M}$ is linearly independent, the representation $\gamma_n = a_n \delta(x_n) + b_n \delta(y_n)$ is unique.
\end{enumerate}
\end{remark}

\begin{proof} The fact that exactly one of the three cases can occur as a consequence of Lemma~\ref{WeaklyClosed}.
	Let us prove $(i)$.  Assume that $|\supp(\gamma)| = 2$ and $\gamma =a\delta(p) + b \delta(q)$. By Lemma~\ref{PrepmaintheoremNew}, there exist $(x_n)_n  , (y_n)_n \subset M$ such that $x_n, y_n \in \supp(\gamma_n)$ for every $n \in \N$, $\lim_n\,d(x_n,p)=0$ and $\lim_n \,d(y_n,q)=0$. There exists $N \in \N$ large enough such that for every $n \geq N$, $d(p , x_n) < \frac{1}{4} d(p,q)$ and $d(q , y_n) < \frac{1}{4} d(p,q)$. In particular $x_n \neq y_n$ and therefore $|\supp(\gamma_n)| = 2$ for every $n \geq N$. Thus we may write $\gamma_n  = a_n \delta(x_n) + b_n \delta(y_n)$ with possibly $a_n =0$ or $b_n=0$ if $n \leq N$. Now we use Lemma~\ref{LemmaConstrctionLipF} to define a map $h \in \Lip_0(M,\R)$ such that $h \equiv 1$ on $B(p, \frac 14 d(p,q))$ while $h \equiv 0$ outside of the ball $B(p , \frac 1 2 d(p,q))$. 
	We readily obtain that $a_n = \<h , \gamma_n \>$ for $n \geq N$ and 
	$$\lim\limits_{n \to + \infty } a_n = \lim\limits_{n \to + \infty } \< h , \gamma_n \> = \< h , \gamma \> =a.$$
	A similar argument can be used to prove that $b_n \to b$.
	\medskip
	
	Let us prove $(ii)$. Assume that $|\supp(\gamma)| = 1$ and $\gamma =a\delta(p)$. By Lemma~\ref{PrepmaintheoremNew}, we can write $\gamma_n = a_n \delta(x_n) + b_n \delta(y_n)$ with $x_n \in \supp(\gamma_n)$ and $x_n \to p$ (and possibly $y_n=0$). We now distinguish two cases:
	\smallskip
	
	- If $y_n \to p$, we let $\ep:=d(p,0)/2$. Then there exists $n_0 \in \N$ such that, for every $n \geq n_0$, $d(x_n,p)< \ep/2$ and $d(y_n,p)< \ep/2$. We use Lemma~\ref{LemmaConstrctionLipF} to define $h \in \Lip_0(M)$ such that $h(p) \equiv 1$ on $B(p,\ep/2)$ and $h \equiv 0$ outside of $B(p, \ep)$. We then observe that 
		$$ \forall n \geq n_0, \quad a_n + b_n = \langle h , \gamma_n \rangle \to \langle h , \gamma \rangle = a.  $$
		This implies that
		$$a_n(\delta(x_n) -\delta(y_n)) = a_n\delta(x_n)+b_n\delta(y_n) -(a_n+b_n)\delta(y_n) \longrightarrow a \delta(p) - a\delta(p) = 0. $$
		Hence $\|a_n(\delta(x_n) -\delta(y_n))\| = |a_n|d(x_n , y_n) \to 0$ (weak and norm convergence coincide in $\FS_2(M)$ thanks to Theorem~\ref{ThmNormToWeak}). Similarly $b_nd(x_n ,y_n) \to 0$.
	\smallskip
	
	- If $y_n \not\to p$, there exist $\ep >0$ and a subsequence $(y_{n_k})_k$ such that $d(y_{n_k},p) > \ep >0$ for every $k \in \N$. Then we define the same kind of map $h$ as above, that is $h(p) \equiv 1$ on $B(p,\ep/2)$ and $h \equiv 0$ outside of $B(p, \ep)$. Observe that
		$$ \forall k \in \N, \quad a_{n_k} = \langle h , \gamma_{n_k} \rangle \to \langle h , \gamma \rangle = a.  $$
		This implies that $b_{n_k} \delta(y_{n_k}) = \gamma_{n_k} - a_{n_k} \delta(x_{n_k}) \to 0$ in the norm topology. 
	\medskip
	
	Finally, $(iii)$ is a direct consequence of Theorem~\ref{ThmNormToWeak} and Lemma~\ref{lemma:formula}.
\end{proof}

\smallskip

\subsection{Sufficient conditions for compactness}\label{sectionsuffcompac}

In this subsection, we first retrieve, thanks to the results obtained in the previous sections, a sufficient condition for compactness of weighted composition operators that has been obtained in \cite[Theorem~3.3]{GolMa22}. From this result, we will deduce a necessary and sufficient condition for compactness  on bounded and uniformly discrete metric spaces. In fact, a very general characterization on any metric space (and through metric conditions) can be obtained. However, the statement is not as nice as the one in \cite{GolMa22} 
since it turns out that having $f$ with a relatively compact range is of great use for one implication. For that reason, and because we will not use it for the rest of the paper, we postpone its statement and its proof to Appendix~\ref{AnnexeA}.
Recall the notation introduced in Section~\ref{bounded}:
\begin{align*}
	\sigma(x,y) &= \frac{d(f(x),f(y))}{d(x,y)}(s(x,y)|w(x)|+s(y,x)|w(y)|),\\
	\tau(x,y)   &= \frac{|w(x)-w(y)|}{d(x,y)} \min\{d(f(x),0) , d(f(y),0) \}, 
\end{align*}
where $s(x,y) = 1$ if $d(f(x),0)\geq d(f(y),0)$, and $s(x,y) = 0$ otherwise.

\begin{theorem}[\protect{\cite[Theorem~3.3]{GolMa22}}] \label{CaracCompact}
	Let $M,N$ be pointed metric spaces.  Let $w : M \to \C$ and $f : M \to N$ be any maps such that $f(0)=0$ or $w(0)=0$. Assume that $w\widehat{f}$ is a bounded operator and that $f(\text{coz}(w))$ is totally bounded. 
	Then $w\widehat{f}$ and $wC_f = (w\widehat{f})^*$ are compact if and only if 
	$$ \lim \sigma(x,y) = 0 \text{ whenever } d(f(x) , f(y)) \to 0, \text{ and}$$ 
	$$ \lim \tau(x,y) = 0 \text{ whenever } \min( d(f(x) ,0) , d(f(y) , 0) )  \to 0.$$
\end{theorem}

\begin{proof} Without loss of generality, we can assume that $N$ is complete, see subsection \ref{remarkcomplete}.

"$\implies$" For this implication, we do not need the assumption ``$f(\text{coz}(w))$ is totally bounded".
	Assume that $w\widehat{f}$ is compact. Let $(x_n)_n, (y_n)_n$ be sequences in $M$ such that $x_n \neq y_n$ for every $n$ and $d(f(x_n) , f(y_n)) \to 0$. By assumption and Proposition \ref{caracCompactWeight}, the sequence $(\gamma_n)_n$ defined by 
	$$ \gamma_n = \dfrac{w(x_n)\delta(f(x_n)) - w(y_n)\delta(f(y_n))}{d(x_n,y_n)} $$
	has a convergent subsequence. So there exists $\gamma \in \FS_2(N)$ and an increasing sequence of integers $(n_k)_k$ such that $\gamma_{n_k} \to \gamma$. Since $d(f(x_n) , f(y_n)) \to 0$, notice that we actually have that $\gamma \in \FS_1(N)$ thanks to Lemma~\ref{SequencesInFS2}~$(i)$. If $\gamma = 0$, then $A(x_{n_k} , y_{n_k}) \to 0$, $B(x_{n_k} , y_{n_k}) \to 0$ and $B(y_{n_k} , x_{n_k}) \to 0$ thanks to Lemma~\ref{SequencesInFS2}~$(iii)$. By Lemma~\ref{inequalities}, we have that 
	$$\sigma(x_{n_k} , y_{n_k}) \leq 2\Big(A(x_{n_k} , y_{n_k})+\max\big(B(x_{n_k} , y_{n_k}) , B(y_{n_k} , x_{n_k})\big)\Big) \to 0.$$ 
	The argument above shows that for every subsequence of $(\sigma(x_n , y_n))_n$, we may extract a further subsequence which converges to 0. This implies that the sequence itself must converge to 0. 
	Next, if $|\supp(\gamma)| = 1$, we write $\gamma = a \delta(p)$. By Lemma~\ref{SequencesInFS2}~$(ii)$, $f(x_{n_k}) \to p$, $f(y_{n_k}) \to p$,  
	$$ \frac{|w(x_{n_k})|}{d(x_{n_k}, y_{n_k})}d(f(x_{n_k}), f(y_{n_k})) \to 0 \quad \text{and} \quad \frac{|w(y_{n_k})|}{d(x_{n_k}, y_{n_k})}d(f(x_{n_k}), f(y_{n_k})) \to 0 . $$
	This obviously implies that $\sigma(x_{n_k} , y_{n_k}) \to 0$ and so $\sigma(x_n,y_n) \to 0$ for the same reasons explained before.
	\smallskip
	
	We proceed similarly if  $(x_n)_n, (y_n)_n$ are sequences in $M$ such that $x_n \neq y_n$ for every $n$ and $\min( d(f(x_n) ,0) , d(f(y_n) , 0) ) \to 0$. Indeed, we consider the same sequence $(\gamma_n)_n \subset \FS_2(N)$ which admits a convergent subsequence (still denoted the same way for simplicity). As above, the limit $\gamma$ must contain at most one point in its support. If $\gamma = 0$ then $A(x_{n} , y_{n}) \to 0$, $B(x_{n} , y_{n}) \to 0$ and $B(y_{n} , x_{n}) \to 0$ by Lemma~\ref{SequencesInFS2}~$(iii)$ and $\sigma(x_n , y_n) \to 0$ by the same reasoning as above. The conclusion now follows from the following inequality proved in Lemma~\ref{inequalities}:
	$$ \tau(x_n ,y_n) \leq A(x_n , y_n) + \sigma(x_n ,y_n) \to 0.$$
	Finally if $\gamma = a \delta(p)$ then we may assume without loss of generality that $f(x_n) \to p$ while $f(y_n) \to 0$. Again Lemma~\ref{SequencesInFS2}~$(iii)$ implies that (passing to a subsequence if necessary) 
	$$\frac{w(x_n)}{d(x_n , y_n)} \to a \quad \text{and} \quad  \frac{w(y_n)}{d(x_n ,y_n)} d(f(y_n) , 0) \to 0.$$
	Now the conclusion follows from the triangle inequality: for $n$ large enough,
	\begin{align*} 
		\frac{|w(x_n) - w(y_n)|}{d(x_n,y_n)}\min(d(f(x_n),0) &, d(f(y_n),0))   \\
		 &	\leq \frac{|w(x_n)|}{d(x_n,y_n)}d(f(y_n) , 0)) + 	\frac{|w(y_n)| }{d(x_n,y_n)}d(f(y_n) ,0) \to 0. 
	\end{align*}
	\medskip

	"$\impliedby$" Let $(\gamma_n)_n$ be a sequence such that for every $n \in \N$,
	$$ \gamma_n = \dfrac{w(x_n)\delta(f(x_n)) - w(y_n)\delta(f(y_n))}{d(x_n,y_n)} .$$
    Assume first that for every $n$, $x_n, y_n \in \text{coz}(w)$.
    Since $N$ is assumed to be complete, $f(\text{coz}(w))$ is relatively compact in $N$.
    Then, up to passing to a subsequence, we may assume that $f(x_n) \to p$ while $f(y_n) \to q$. We will distinguish three cases.
    \smallskip
    
    If $p \neq 0$, $q\neq 0$ and $p \neq q$ then Theorem~\ref{Poids}~$(v)$ implies that the sequences $$\Big(\frac{|w(x_n) - w(y_n)|}{d(x_n , y_n)}\Big)_n \text{ and }
    \Big(\frac{s(x_n,y_n)|w(w_n)| + s(x_n , y_n)|w(y_n)|}{d(x_n , y_n)}\Big)_n$$
    are bounded. This yields that $\Big(\frac{w(x_n)}{d(x_n , y_n)}  \Big)_n$ and $\Big(\frac{w(y_n)}{d(x_n , y_n)}  \Big)_n$ are bounded as well, so we may extract subsequences which converge to $a \in \C$ and $b \in \C$ respectively. Therefore $(\gamma_n)_n$ has a subsequence which converges to $a \delta(p) + b \delta(q)$. 
    \smallskip

    If $p \neq 0$ and $q=0$, then by assumption $\tau(x_n , y_n) \to 0$. Now Theorem~\ref{Poids}~$(v)$ implies that 
    $\Big( \frac{d(f(x_n) , f(y_n))}{d(x_n ,y_n)}|w(x_n)| \Big)_n$
    is bounded. We easily deduce that the sequence $\Big(\frac{w(x_n)}{d(x_n , y_n)}  \Big)_n$ must be bounded as well, so we may extract a subsequence (still denoted the same way to simplify the notation) which converges to some $a\in \C$. Thus $\Big(\frac{w(x_n)}{d(x_n , y_n)}\delta(f(x_n)\Big)_n$ converges to $a \delta(p)$ and:
    \begin{align*}
        |w(y_n)|\frac{d(f(y_n) , 0)}{d(x_n , y_n)} \leq \frac{|w(x_n)|}{d(x_n,y_n)}d(f(y_n ),0) + \frac{|w(x_n) - w(y_n)|}{d(x_n ,y_n)}d(f(y_n) , 0) \to 0.
    \end{align*}
    The above inequality yields that $\Big(\frac{w(y_n)}{d(x_n , y_n)}\delta(f(y_n)\Big)_n$ converges to 0, and so $(\gamma_n)_n$ converges to $a \delta(p)$. 
    \smallskip

    Finally let us assume that $p = q$. If $p = q = 0$ then by assumption $\sigma(x_n , y_n) \to 0$ and $\tau(x_n ,y_n) \to 0$, so Lemma~\ref{lemma:formula} together with Lemma~\ref{inequalities} yields that $\gamma_n \to 0$. If $p=q \neq 0$ then $\sigma(x_n , y_n) \to 0$. Furthermore the sequence $\Big( \frac{w(x_n) - w(y_n)}{d(x_n ,y_n)} \Big)_n$ is bounded so we may extract a subsequence (again denoted the same way) which converges to some $a \in \C$. Up to extracting a further subsequence, we may assume that $s(y_n , x_n)= 1$ for every $n$. Now observe that 
    \begin{align*}
        \| \gamma_n - a \delta(p) \| \leq \Big\| \frac{w(x_n) - w(y_n)}{d(x_n , y_n)}\delta(f(x_n)) - a \delta(p)\Big\| + |w(y_n)| \Big\| \frac{\delta(f(x_n)) - \delta(f(y_n))}{d(x_n ,y_n)} \Big\| \to 0.
    \end{align*}
    Next, if there is $(n_k)_k \subset \N$ increasing such that $x_{n_k}, y_{n_k} \notin \text{coz}(w)$, then $\gamma_{n_k} = 0$ for every $k$ so $(\gamma_n)_n$ has a convergent subsequence. Finally, assume (up to a subsequence), that for every $n$, $x_n \in \text{coz}(w)$ and $y_n \notin \text{coz}(w)$. In that case, $\gamma_n = \frac{w(x_n)}{d(x_n,y_n)} \delta(f(x_n))$. Again, by compactness, we may assume that $f(x_n) \to p$. Now, with a similar reasoning as before, we show that $m_n \to 0$ if $p=0$ and that (a subsequence of) $(m_n)_n$ converges to $a\delta(p)$ for some $a\in \mathbb{C}$ if $p\neq 0$. This completes the proof.
\end{proof}
\medskip 

\begin{remark}~
\begin{enumerate}[itemsep=5pt]
    \item It is worth noticing that ``$w\widehat{f}$ compact'' does not imply that $f(\{x \in M  : w(x) \neq 0\})$ is totally bounded. To see this, take $w \equiv 1$ in \cite[Example 2.10]{ACP21}. In fact, this example shows that for some $\alpha >0$, $f(\{x \in M  : |w(x)|> \alpha\})$ can be unbounded. However, when $M$ is bounded, $f(\{x \in M  : |w(x)|> \alpha\})$ is totally bounded for every $\alpha >0$ whenever $w\widehat{f}$ compact.
    \item Assuming moreover that $w$ is a Lipschitz map, one can deduce the next simpler statement (which corresponds to \cite[Theorem~4.3]{DA}): $w\widehat{f}$ and $wC_f$ are compact if and only  
    $$  |w(x)| \frac{d(f(x) , f(y))}{d(x,y)} \to 0$$
    whenever $d(f(x),f(y)) \to 0$.
\end{enumerate}
\end{remark}

The previous theorem provided, roughly speaking, a characterization of compact weighted operators when the range space is compact. On the opposite side, the next result gives a characterization in the case when the domain space $M$ is uniformly discrete, that is there is some $\theta >0$ such that $\forall x \neq y \in M$, $d(f(x) ,f(y)) \geq \theta$.  

\begin{proposition} \label{CaracCompactUDB}
	Let $M,N$ be pointed metric spaces such that $M$ is uniformly discrete and bounded.  Let $w : M \to \C$ and $f : M \to N$ be any maps such that $f(0)=0$ or $w(0)=0$. Assume that $w\widehat{f}$ is a bounded operator. Then $w\widehat{f}$ and $wC_f = (w\widehat{f})^*$ are compact if and only if:
    \begin{enumerate}[$(i)$]
        \item $f(\{x \in M  : |w(x)| > \alpha \})$ is totally bounded for every $\alpha >0$, and 
        \item $ \lim w(x)d(f(x) , 0) = 0$ whenever $w(x) \to 0$ or $w(x) \to +\infty$.
    \end{enumerate}
\end{proposition}

\begin{proof} As in the proof of Theorem \ref{CaracCompact}, one can assume that $N$ is complete. Hence,
    thanks to Lemma~\ref{SequencesInFS2}, $(i)$ and $(ii)$ are equivalent to $\{w(x)\delta(f(x))  : x \in M\}$ being relatively compact. Therefore, we only have to prove the next equivalence:
    $$w\widehat{f} \text{ is compact } \iff \{w(x)\delta(f(x))  : x \in M\} \text{ is relatively compact.}$$
    First, since $M$ is uniformly discrete and bounded, the set $\{w(x)\delta(f(x))  : x \in M\}$ is relatively compact if and only if $\{\frac{w(x)}{d(x,0)}\delta(f(x))  : x \in M\}$ is so. Thus ``$\implies$'' follows from Proposition~\ref{caracCompactWeight}. We now concentrate on ``$\impliedby$''. Let $(x_n)_n, (y_n)_n$ be sequences in $M$ such that $x_n \neq y_n$ and let $(\gamma_n)_n$ be the sequence defined by 
	$$ \gamma_n = \dfrac{w(x_n)\delta(f(x_n)) - w(y_n)\delta(f(y_n))}{d(x_n,y_n)} .$$
    Since $M$ is uniformly discrete and bounded, $\theta \leq d(x_n , y_n) \leq \diam(M)$ for every $n \in \N$. Therefore the sequence $(d(x_n , y_n))_n$ admits a subsequence which converges to some $\rho \in [\theta , \diam(M)]$. By assumption, the sequences $(w(x_n)\delta(f(x_n)))_n$ and $(w(y_n)\delta(f(y_n)))_n$ both have a convergent subsequence, say to $a \delta(p)$ and $b \delta(q)$ respectively. Hence $(\gamma_n)_n$ has a subsequence which converges to $\frac{a}{\rho} \delta(p) + \frac{b}{\rho} \delta(q)$. 
\end{proof}

Unfortunately, one cannot remove the boundedness assumption in the above proposition since the domain space in \cite[Example 2.10]{ACP21} is uniformly discrete, unbounded and $f(\{x \in M  : |w(x)|> \alpha\}) =M$ is unbounded.
\bigskip

\subsection{A link with the non weighted case.}\label{sectioncompacw=1}

First of all, let us extend \cite[Theorem A]{ACP21} to the complex scalars setting. 
Assume that $w=1$. Let $f : M \to N$ be a base point preserving Lipschitz map. 
We can see $\widehat{f}$ either as an operator from $\F(M,\C)$ to $\F(N,\C)$ or from $\F(M,\R)$ to $\F(N,\R)$. To avoid confusion, we will write $\widehat{f}_\R$ for the latter operator. Now from the isometric isomorphism given in Proposition~\ref{pten}, one can see that $\widehat{f} : \F(M,\C) \to \F(M,\C)$ is conjugated to the operator
$$\begin{array}{ccccc}
T & : & \F(M,\R) \pten \C & \to & \F(N,\R) \pten \C \\
 & & \gamma_1 \otimes 1 + \gamma_2 \otimes i & \mapsto & \widehat{f}_{\R}(\gamma_1) \otimes 1 + \widehat{f}_\R(\gamma_2) \otimes i .\\
\end{array}$$
Hence $\widehat{f} : \F(M,\C) \to \F(M,\C)$ is compact if and only if $\widehat{f}_\R : \F(M,\R) \to \F(M,\R)$ is compact. In particular, the characterization of compactness in \cite[Theorem A]{ACP21} works as well in the complex case:

\begin{theorem}\label{ThmAcomplex}
	Let $M,N$ be complete pointed metric spaces, and let $f : M \to N$ be a base point preserving Lipschitz mapping. Then $\widehat{f} : \F(M, \C) \to \F(N, \C)$ is compact if and only if the following assertions are satisfied:
	\begin{enumerate}
		\item[$(P_1)$]  For every bounded subset $S \subset M$, $f(S)$ is totally bounded in $N$;
		\item[$(P_2)$] $f$ is uniformly locally flat, that is, 
		$$ \lim\limits_{d(x,y) \to 0} \dfrac{d(f(x),f(y))}{d(x,y)} =0;$$
		\item[$(P_3)$] For every $(x_n,y_n)_n \subset \widetilde{M} : = \{(x,y) \in M \times M \; | \; x \neq y\}$ such that \\
		$\lim\limits_{n \to \infty} d(x_n,0) = \lim\limits_{n \to \infty} d(y_n,0) = \infty$, either 
		\smallskip
		
		\begin{itemize}
			\item $(f(x_n), f(y_n))_n$ has an accumulation point in $N \times N$, or
			\item $\underset{n\to+\infty}{\liminf}\,\dfrac{d(f(x_n),f(y_n))}{d(x_n,y_n)}=0$.
		\end{itemize}
	\end{enumerate}
\end{theorem}

\begin{remark}\label{strongradiallyflat}
In fact, it is straightforward to check that all the ingredients to prove \cite[Theorem A]{ACP21} work as well in the complex case. In particular, \cite[Remark 2.7]{ACP21} applies, and if $\widehat{f} : \F(M, \C) \to \F(N, \C)$ is compact, then we must have, in the case of an unbounded metric space $M$,
	$$
	\underset{d(x,y)\to +\infty}{\lim} \ \dfrac{d(f(x),f(y))}{d(x,y)} = 0.
	$$
\end{remark}

We retrieve \cite[Theorem 1.1]{Vargas1} as a direct consequence of the previous theorem.

\begin{corollary}
\label{w=1Lipversion}
Let $M,N$ be complete metric spaces and let $f : M \to N$ be a Lipschitz map. The following assertions are equivalent:
\begin{enumerate}[itemsep = 4pt]
 		\item[$(i)$] $C_f: \Lip(N) \to \Lip(M)$ is compact;
        \item[$(ii)$] $f$ is uniformly locally flat and $f(M)$ is totally bounded in $N$.
\end{enumerate}
\end{corollary}

\begin{proof}
In this proof, we use the notation introduced in Section \ref{Section1.2} in the special case $w=1$. By Proposition~\ref{Lip0toLip}, $C_f: \Lip(N) \to \Lip(M)$ is compact if and only if $C_{f^e} : \Lip_0(N^e) \to \Lip_0(M^e)$ is compact, and, by duality, if and only if $\widehat{f^e} : \F(M^e) \to \F(N^e) $ is compact. Since $M^e$ is bounded, by Theorem \ref{ThmAcomplex}, the compactness of $\widehat{f^e}$ is equivalent to the fact that $f^e$ is uniformly locally flat and $f^e(M^e)$ is totally bounded in $N^e$. Now, from the definition of the metric defined on $M^e$ in Section \ref{Section1.2}, it is a simple exercise to check that this is equivalent to $(ii)$.
\end{proof}

We conclude the section by further exploring the connections between the maps $w\widehat{f}$ and $\varphi : x \in M \mapsto w(x)\delta(f(x)) \in \F(N)$. Indeed, we proved in Theorem~\ref{Poids} that $w\widehat{f}$ is bounded if and only if $\varphi$ is Lipschitz. 
So it might be tempting to say that $w\widehat{f}$ is compact if and only if $\widehat{\varphi}$ is so. This is actually not really accurate. 
In fact, the unique extension of $\varphi$ to $\F(M)$ using Theorem~\ref{thm1} verifies $\overline{\varphi} = w \widehat{f}$. So $w\widehat{f}$ is compact if and only if $\overline{\varphi}$ is compact. Still, one can relate $\overline{\varphi}$ and $\widehat{\varphi}$ by a simple diagram chasing argument:
$$\def\commutatif{\ar@{}[rd]|{\circlearrowleft}}
	\xymatrix{
		M \ar[r]^{\varphi} \ar@{^{(}->}[d]_{\delta_M} & \F(N) \ar[d]_{\delta_{\F(N)}} \\
		\F(M) \ar[ru]^{\overline{\varphi}} \ar[r]^{\widehat{\varphi}} & \F\big(\F(N)\big) \ar@/_/[u]_{\beta}
	}
	$$
	In this diagram, $\beta$ is the quotient map of norm 1 such that $\beta(\sum a_i \delta_{\F(N)}(\gamma_i)) = \sum a_i \gamma_i$. In particular $\beta \circ \delta_{\F(N)} = Id$ and
	\begin{align*}
		&\beta \circ \widehat{\varphi}(\delta_M(x)) = \beta (\delta_{\F(N)}(\varphi(x))) = \varphi(x); \\
		&\overline{\varphi}(\delta_M(x)) = \varphi(x).
	\end{align*}
	Hence $\beta \circ \widehat{\varphi} = \overline{\varphi}$. Consequently, if $\widehat{\varphi}$ is compact, then $\overline{\varphi}$ is compact. Since Theorem~\ref{ThmAcomplex} characterizes the maps $\varphi$ such that $\widehat{\varphi}$ is compact, one can deduce the next corollary (we assume that $M$ is bounded only to keep a nicer form of the statement).

\begin{corollary} \label{thmA} 
 Let $M,N$ be pointed metric spaces with $M$ bounded. Assume that $\varphi : x \in M \mapsto w(x)\delta(f(x)) \in \F(N)$ is Lipschitz. If 
	\begin{itemize}
		\item $\varphi(M)$ is totally bounded, and
		\item $\varphi$ is uniformly locally flat, 
	\end{itemize}
	then $w\widehat{f}$ and $wC_f = (w\widehat{f})^*$  are compact operators.
\end{corollary}

\begin{remark}
    It does not hold that $\widehat{\varphi}$ is compact whenever $\overline{\varphi}$ is compact. Here is an example: Let $M=\{0\}\cup[1,+\infty[$, $N=\{0,1\}$. Consider the map $f : M \to N$ defined by $f(0)=0$ and $f(x)=1$ otherwise. Finally define the weight $w : M \to [0,+\infty[$ by: $w(x) = x$. Then it is easy to see that $\varphi$ is Lipschitz using Theorem~\ref{Poids}. Obviously $\overline{\varphi}$ is compact because it has finite rank. But, $\widehat{\varphi} : \delta(x) \in \F(M)\setminus\{0\} \mapsto \delta(w(x)\delta(1))$ is not compact since $\varphi :x \in M \mapsto x \delta(1) \in \F(N)$ is clearly not uniformly locally flat. In fact, $\varphi$ can be seen as the identity map from $M$ to (some space isometric to) $M$.
\end{remark}

\appendix

\section{Sequential criterion for compactness of weighted Lipschitz operators}\label{AnnexeA}

Let $M,N$ be pointed metric spaces.
Let $w : M \to \C$ and $f : M \to N$ be any maps such that $f(0)=0$ or $w(0)=0$. 
For convenience of the reader, let us recall the notation:
\begin{align*}
	A(x,y) &= \frac{1}{d(x,y)}|w(x)d(f(x),0) - w(y)d(f(y),0)|,  \\
	B(x,y) &= \frac{1}{d(x,y)}|w(x)d(f(x),0) - w(y)(d(f(x),0) - d(f(x),f(y))|.
\end{align*}

\begin{theorem}\label{CaracCompactgeneral}
	Let $M,N $ be complete pointed metric spaces. If $w\widehat{f} : \F(M) \to \F(N)$ is a bounded operator then $w\widehat{f}$ and $wC_f = (w\widehat{f})^*$ are compact if and only if the assertions below are satisfied for every sequences $(x_n)_n  , (y_n)_n \subset M$. To simplify the notation we write $a_n := w(x_n) d(x_n , y_n)^{-1}$ and $b_n := w(y_n) d(x_n , y_n)^{-1}$.
	\begin{enumerate}[leftmargin =*, itemsep=5pt]
		\item If $a_n \to a \neq 0$ and $b_n \to b \neq 0$, then\\
\noindent - either we have $A(x_n,y_n) \to 0$, $B(x_n,y_n) \to 0$ and $B(y_n,x_n) \to 0$;

\noindent - or there exist $p , q \in N$ and $(n_k)_k \subset \N$ increasing such that $f(x_{n_k}) \to p$, $f(y_{n_k}) \to q$.

		\item If $a_n \to a \neq 0$ and $|b_n| \to 0$ or $+\infty$ (resp. $|a_n| \to 0$ or $+\infty$ and $b_n \to b\neq 0$), then\\
  \noindent - either we have $A(x_n,y_n) \to 0$, $B(x_n,y_n) \to 0$ and $B(y_n,x_n) \to 0$;

\noindent - or there exist $p \in N$, $(n_k)_k \subset \N$ increasing such that $f(x_{n_k}) \to p$ and $d(f(y_{n_k}), 0) b_{n_k}\to 0$ (resp. $f(y_{n_k}) \to p$ and $d(f(x_{n_k}), 0) a_{n_k}\to 0$).

		\item If $a_n \to 0$ and $b_n \to 0$, then $A(x_n , y_n) \to 0$, $B(x_n , y_n) \to 0$ and $B(y_n , x_n) \to 0$.
		\item If $|a_n| \to +\infty$, $|b_n| \to +\infty$ and $|a_n- b_n| \to 0$ or $+\infty$, then $A(x_n , y_n) \to 0$, $B(x_n , y_n) \to 0$ and $B(y_n , x_n) \to 0$.
		\item If $|a_n| \to +\infty$, $|b_n| \to +\infty$ and $a_n - b_n \to a \neq 0$, then\\
\noindent - either $A(x_n , y_n) \to 0$, $B(x_n , y_n) \to 0$ and $B(y_n , x_n) \to 0$;\\
\noindent - or there exist $p\in N$, $(n_k)_k \subset \N$ increasing such that $f(x_{n_k}) \to p$, $f(y_{n_k}) \to p$ and $~b_{n_k}d(f(x_{n_k}), f(y_{n_k})) \to 0$.

	\end{enumerate}
\end{theorem}

\begin{proof}
We start by showing that the conditions $(1)$ to $(5)$ are necessary. Assume that $w\widehat{f}$ is compact. Let $(x_n)_n, (y_n)_n \subset M$ be sequences such that, for every $n$, $x_n \neq y_n$. Let, for any $n$,
$$m_n := \frac{(w(x_n)\delta(f(x_n)) - w(y_n)\delta(f(y_n)))}{d(x_n,y_n)} = a_n \delta(f(x_n)) - b_n \delta(f(y_n)).$$
By Proposition \ref{caracCompactWeight}, the sequence $(m_n)_n$ has a convergent subsequence. To simplify the notation, we will assume that the whole sequence $(m_n)_n$ converges to $\gamma \in \F(M)$. Note that
$
m_n  \in \FS_2(N)
$
so, by Lemma \ref{WeaklyClosed}, there exist $\alpha, \beta\in \C$, $p,q \in N$, such that $\gamma = \alpha\delta(p) + \beta\delta(q)$. 
\begin{enumerate}[leftmargin =*, itemsep=5pt]
\item Assume that $a_n \to a \neq 0$ and $b_n \to b \neq 0$. If $|\supp(\gamma)| = 2$, by Lemma~\ref{SequencesInFS2}, there exist $p,q \in N, p\neq q$, such that $f(x_n) \to p$ and $f(y_n) \to q$.
If $|\supp(\gamma)| = 1$, we can assume, by Lemma~\ref{PrepmaintheoremNew}, that there is a subsequence $(f(x_{n_k}))_k$ converging to $p\neq 0$. This implies that $b_{n_k}\delta(f(y_{n_k})) = a_{n_k}\delta(f(x_{n_k}))- m_{n_k}$ converges to some $c\delta(r)$. But since $b_{n_k} \to b$, this forces $f(y_{n_k}) \to r$. Finally, if $|\supp(\gamma)| = 0$, $\| m_n \| \to  0$ so, by Lemma~\ref{lemma:formula}, $A(x_n,y_n) \to 0$, $B(x_n,y_n) \to 0$ and $B(y_n,x_n) \to 0$.
\item Assume that $a_n \to a \neq 0$ and $b_n\to 0$. If $|\supp(\gamma)| = 2$, then, as before and up to a subsequence, we can assume that $f(x_n) \to p$ and $f(y_n)\to q$. This implies that $b_n d(f(y_n), 0) = b_n \| \delta(f(y_n)) \| \to 0$. If $|\supp(\gamma)| = 1$, one can assume that either $f(x_n) \to p$ or $f(y_n)\to p$. But if $f(y_n)\to p$, we have $b_n \delta (f(y_n)) \to 0$ so we must have that $(a_n \delta(f(x_n)))_n$ converges to $a\delta(p)$ and this implies that $(f(x_n))_n$ converges to $p$. On the other hand, if $f(x_n) \to p$, then $b_n\delta(f(y_n)) = a_n\delta(f(x_n))- m_n$ must converge to an element of $\FS_1(N)$. Since $b_n \to 0$, we must have $b_n\delta(f(y_n)) \to 0$.
Finally, if $|\supp(\gamma)| = 0$, $\| m_n \| \to  0$ and we conclude as in the previous case.
\item[(2')] Assume that $a_n \to a \neq 0$ and $|b_n|\to +\infty$. If $|\supp(\gamma)| = 0$, we conclude as in the previous case. If $|\supp(\gamma)| = 2$, we have that $f(x_n) \to p\neq 0$. But then since $b_n\delta(f(y_n)) = a_n\delta(f(x_n))- m_n$ converges and $|b_n| \to +\infty$, this forces $(f(y_n))_n$ to converge to $0$ and contradicts the fact that $|\supp(\gamma)| = 2$. Finally, if $|\supp(\gamma)| = 1$ then we distinguish two cases. In the first case we assume that some subsequence $f(x_{n_k}) \to p \neq 0$. Then the same reasoning as before shows that necessarily $b_{n_k} \delta(f(y_{n_k})) \to 0$. 
Indeed, otherwise a subsequence of $(f(y_{n_k}))_k$ must converge to some $q\neq 0$ and the fact that $|b_{n_k}| \to +\infty$ gives a contradiction. 
The second case is that some subsequence $f(y_{n_k}) \to p \neq 0$. Then, the equality
$$
\delta(f(y_{n_k})) = \frac{a_{n_k}}{b_{n_k}} \delta(f(x_{n_k})) - \dfrac{m_{n_k}}{b_{n_k}}
$$
and the fact that $\dfrac{m_{n_k}}{b_{n_k}} \to 0$ implies that $(\frac{a_{n_k}}{b_{n_k}} \delta(f(x_{n_k})))_k$ goes to $\delta(p)$. This forces $f(x_{n_k}) \to p$ and we readily obtain a contradiction since $\frac{a_{n_k}}{b_{n_k}} \to 0$. So this case cannot happen.

\item Assume that $a_n \to 0$ and $b_n \to 0$. If $|\supp(\gamma)| = 2$, then $(f(x_n))_n$ and $(f(y_n))_n$ converge. But, in that case, using our assumptions, $m_n \to 0$. If $|\supp(\gamma)| = 1$, we can assume that $f(x_n) \to p \neq 0$ so that $a_n \delta(f(x_n)) \to 0$ and hence $b_n \delta(f(y_n)) = a_n \delta(f(y_n)) -m_n$ converges. This implies that $f(y_n) \to p$, and in turn $m_n \to 0$, which is a contradiction. Hence, we necessarily have $|\supp(\gamma)| = 0$, which yields the result.

\item Assume $|a_n| \to +\infty$, $|b_n| \to +\infty$ and $a_n - b_n \to 0$. We cannot have $|\supp(\gamma)| = 2$ thanks to Lemma~\ref{SequencesInFS2}.
If $|\supp(\gamma)| = 1$, then one can assume for instance that $f(x_n) \to p$ where $\gamma =\alpha\delta(p)$, $\alpha\neq 0, p\neq 0$. Now, note that we have
$$
m_n = (a_n - b_n) \delta(f(x_n)) + b_n (\delta(f(x_n)) - \delta(f(y_n))).
$$
where $(a_n - b_n) \delta(f(x_n)) \to 0$ by assumption. Hence, $b_n (\delta(f(x_n)) - \delta(f(y_n))) \to \alpha\delta(p)$. Since $|b_n| \to +\infty$, we must have $\| \delta(f(x_n)) - \delta(f(y_n)) \| \to 0$ so $f(y_n) \to p$. Finally, choose $k\in \Lip_0(N)$ such that $k(p)=k(f(x_n)) = k(f(y_n))=1$ for $n$ large enough. Then
$$
\alpha = \left\langle k, \alpha\delta(p) \right\rangle = \lim\limits_{n} b_n (k(f(x_n)) - k(f(y_n))) = 0,
$$
which is a contradiction. Hence, we must have $|\supp(\gamma)| = 0$ and we get the desired result.
\item[(4')] Assume $|a_n| \to +\infty$, $|b_n| \to +\infty$ and $|a_n - b_n| \to +\infty$. As in the previous case, the case $|\supp(\gamma)| = 2$ cannot happen, and we only have to prove that neither can the case $|\supp(\gamma)| = 1$. Assume that $m_n \to \alpha\delta(p)$ with $\alpha\neq 0$ and $p\neq 0$. In that case, if, say, $f(x_n) \to p \neq 0$, then we must have $f(y_n) \to p$ as well. Otherwise, up to a subsequence, $d(f(y_n), p) \geq \epsilon > 0$, and using Lemma~\ref{LemmaConstrctionLipF} we pick $h\in \Lip_0(N)$ such that $h(p) = h(f(x_n)) = 1$ and $h(f(y_n)) = 0$ for $n$ large enough. Then we get that
$
\left\langle h, m_n \right\rangle = a_n
$
which must converge to $\alpha$, which is not the case. So $f(y_n) \to p$, and with the help of $h$, we get
$$
\left\langle h, m_n \right\rangle = a_n -  b_n
$$
must converge as well, but again, by assumption, it is not the case.
\item Assume $|a_n| \to +\infty$, $|b_n| \to +\infty$, and $a_n - b_n\to a \neq 0$. As before, we can only have $|\supp(\gamma)| = 0$ or $1$.

If $|\supp(\gamma)| = 1$, up to a subsequence we can assume that $f(x_n) \to p \neq 0$ where $\gamma = \alpha\delta(p)$. By choosing an appropriate Lipschitz function, one easily shows that there is a subsequence of $(f(y_n))_n$ converging to $p$. For simplicity, assume that $f(y_n) \to p$. The identity
$$
  m_n = (a_n - b_n) \delta(f(x_n)) + b_n (\delta(f(x_n)) - \delta(f(y_n)))  
$$
implies that $b_n (\delta(f(x_n)) - \delta(f(y_n))) \to \gamma - \alpha \delta(p) = (a-\alpha)\delta(p)$. By choosing a function $h \in \Lip_0(N)$ such that $h(m) = 1$ in a neighborhood of $p$ , we hence obtain, for $n$ large enough,
$$a-\alpha = b_n (h(f(x_n)) - h(f(y_n))) = b_n (1-1) = 0.$$
This implies that $a = \alpha$, so that $|b_n|d(f(x_n), f(y_n)) = \| b_n (\delta(f(x_n)) - \delta(f(y_n))) \| \to 0$.

Finally, if $|\supp(\gamma)| = 0$,  $\| m_n \| \to  0$ and we conclude as before.
\end{enumerate}

\medskip

To conclude, we claim that the conditions $(1)$ to $(5)$ are sufficient. Indeed, by Proposition~\ref{caracCompactWeight}, we only have to prove that $(m_n)_n$ has a convergent subsequence. Since $(a_n)_n$ and $(b_n)_n$ are sequences in $\C$, there is $(n_k)_k \subset \N$ increasing such that $a_{n_k} \to a\in \C$ or $|a_{n_k}| \to +\infty$, $b_{n_k} \to b\in \C$ or $|b_{n_k}| \to +\infty$, and $a_{n_k}  - b_{n_k} \to c \in \C$ or $|a_{n_k}  - b_{n_k}| \to +\infty$. Then one of the cases $(1)$ to $(5)$ must happen and we can find a convergent subsequence to $(m_{n_k})_k$. We leave the details to the reader.
\end{proof}

\begin{example}
Let $\alpha, \beta >0$ 
and define $d_n=n^{-\alpha}$ for every $n\in \N$.
Let $M= \mathbb{N} \cup \{0\}$ with the metric given by $d(n,m)=d_n+d_m$ for $n, m \in \N$ such that 
$n\neq m$, and $d(n,0) = d_n$ otherwise.
Now consider $f : M \to M$ the map defined by $f(0)=0$ and, for every $n\in \mathbb{N}$, $f(n)=n-1$.  Finally define the weight map $w : M \to \mathbb{R}$ by setting $w(0)=0$ and $w(n) = n^{-\beta}$. By \cite[Proposition~1.6]{ACP20}, the linear map given by $\psi : \F(M) \to \ell_1(\mathbb{N})$ defined by $\delta(n) \mapsto d_ne_n$ is an isometric isomorphism (it is stated in the real setting, but it holds true in the complex setting). 
It is moreover easily checked that $w\widehat{f}$ is conjugated to an operator $T : \ell_1(\mathbb{N}) \to \ell_1(\mathbb{N})$ such that $Te_1 = 0$ and 
$$Te_n = \dfrac{d_{n-1}}{d_n} w(n) e_{n-1} = \left(\frac{n}{n-1}\right)^{\alpha} \dfrac{1}{n^{\beta}} e_{n-1}$$
otherwise. That is, $T$ is a weighted backward shift. It is clear that $T$, and hence $w \widehat{f}$, is compact whenever $\beta > 0$.
\smallskip

We will see that any of the items $(1)$ to $(5)$ from the previous theorem can happen. We keep using the notation $a_n = w(x_n)d(x_n,y_n)^{-1}$ and $b_n = w(y_n)d(x_n,y_n)^{-1}$ for some $x_n$ and $y_n$ specified below.
\begin{enumerate}[itemsep = 5pt]
\item If $\alpha=\beta$, $x_n = n$ and $y_n =  n-1$, we have $a_n,  b_n  \to \frac{1}{2}$ as $n \to +\infty$ and $f(x_n), f(y_n) \to 0$.
\item If $\alpha = \beta$, $x_n  = n$ and $y_n = n^2$, we have $a_n \to 1$ and $b_n \to 0$. Then, we see that $f(x_n) \to 0$ and $d(f(y_n),0)b_n \to 0$.
\item[$(2')$] If $\alpha = 2\beta$, $x_n  = n^2$ and $y_n = n$, we have $a_n \to 1$ and $b_n \to +\infty$. Then, we see that $f(y_n) \to 0$ and $d(f(x_n),0)w(x_n)d(x_n, y_n)^{-1} \to 0$.
\item If $\beta > \alpha$, $x_n = n$ and $y_n =  n-1$, $a_n,  b_n  \to 0$ and one can check that $A(x_n,y_n), B(x_n,y_n)$ and $B(y_n,x_n)$ tend to $0$.
\item If $\alpha > \beta +1$, $x_n=n$ and $y_n = n-1$, we have $a_n, b_n \to +\infty$, $w(x_n) d(x_n , y_n)^{-1} - w(y_n) d(x_n , y_n)^{-1}  \sim \frac{-\beta}{2n^{\beta-\alpha  +1}} \to 0$ and $A(x_n,y_n), B(x_n,y_n)$ and $B(y_n,x_n)$ tend to $0$.
\item If $\alpha = \beta +1$, $x_n=n$ and $y_n = n-1$, we have $a_n, b_n \to +\infty$, $w(x_n) d(x_n , y_n)^{-1} - w(y_n) d(x_n , y_n)^{-1} \to \frac{-\beta}{2}$ and $A(x_n,y_n), B(x_n,y_n)$ and $B(y_n,x_n)$ tend to $0$.
\end{enumerate}
\end{example}

\noindent \textbf{Acknowledgements :} The third author was supported by the French ANR project No$.$ ANR-20-CE40-0006.

\end{document}